\documentclass[aap]{imsart}

\RequirePackage{amsthm,amsmath,amsfonts,amssymb,mathtools}
\RequirePackage[numbers]{natbib}
\RequirePackage[colorlinks,citecolor=blue,urlcolor=blue]{hyperref}
\RequirePackage{graphicx}
\usepackage{bbm}
\usepackage{enumerate}
\usepackage[dvipsnames]{xcolor}
\usepackage{hyperref}
\usepackage{tikz}
\usepackage{marginnote}
\usetikzlibrary{arrows}

\startlocaldefs
\numberwithin{equation}{section}
\theoremstyle{plain}

\newtheorem{theorem}{Theorem}[section]
\newtheorem{lemma}[theorem]{Lemma}
\newtheorem{corollary}[theorem]{Corollary}
\newtheorem{proposition}[theorem]{Proposition}

\theoremstyle{remark}
\newtheorem{definition}[theorem]{Definition}

\newtheorem{remark}[theorem]{Remark}

\renewcommand{\P}{\mathbb{P}}
\newcommand{\E}{\mathbb{E}}
\newcommand{\R}{\mathbb{R}}
\newcommand{\N}{\mathbb{N}}
\newcommand{\Z}{\mathbb{Z}}

\newcommand{\ind}{\mathbbm{1}}
\newcommand{\Block}{\mathtt{Block}}
\newcommand{\block}{\mathrm{block}}
\newcommand{\sspace}{\mathrm{space}}
\newcommand{\ttime}{\mathrm{time}}
\newcommand{\dens}{\mathrm{dens}}

\DeclareMathOperator{\Supp}{Supp}
\DeclareMathOperator{\Var}{Var}
\DeclarePairedDelimiter{\norm}{\lVert}{\rVert}
\DeclarePairedDelimiter{\abs}{\lvert}{\rvert}

\endlocaldefs

\begin{document}

\begin{frontmatter}
\title{Survival and complete convergence for a branching annihilating random walk}
\runtitle{Survival and complete convergence for a BARW}

\begin{aug}
\author[A]{\fnms{Matthias}~\snm{Birkner}\ead[label=e1]{birkner@mathematik.uni-mainz.de}},
\author[A,B]{\fnms{Alice}~\snm{Callegaro}\ead[label=e2]{alice.callegaro@tum.de}}
\author[C]{\fnms{Jiří}~\snm{Černý}\ead[label=e3]{jiri.cerny@unibas.ch}\orcid{0000-0001-7723-9245}}
\author[B]{\fnms{Nina}~\snm{Gantert}\ead[label=e4]{nina.gantert@tum.de}\orcid{0000-0003-0811-3651}}
\and
\author[A,C]{\fnms{Pascal}~\snm{Oswald}\ead[label=e5]{pascalamadeus.oswald@unibas.ch}}
\address[A]{
Johannes Gutenberg-Universit\"{a}t Mainz\printead[presep={,\ }]{e1}}

\address[B]{
Technical University of Munich\printead[presep={,\ }]{e2,e4}}

\address[C]{
University of Basel\printead[presep={,\ }]{e3,e5}}
\end{aug} 

\begin{abstract}
  We study a discrete-time branching annihilating random walk (BARW) on
  the $d$-dimensional lattice.  Each particle produces a Poissonian
  number of offspring with mean $\mu$ which independently move to a
  uniformly chosen site within a fixed distance $R$ from their parent's
  position. Whenever a site is occupied by at least two particles, all
  the particles at that site are annihilated. We prove that for any
  $\mu>1$ the process survives when $R$ is sufficiently large. For fixed
  $R$ we show that the process dies out if $\mu$ is too small or too
  large. Furthermore, we exhibit an interval of $\mu$-values for which
  the process survives and possesses a unique non-trivial ergodic
  equilibrium for $R$ sufficiently large.  We also prove complete
  convergence for that case.
\end{abstract}



\end{frontmatter}
\tableofcontents

\section{Introduction}

Branching random walks are well-known models for populations evolving in space. In these systems individuals are represented as particles which reproduce and move randomly in space, independently for
different families. For instance, the children may take
i.i.d.~displacements from their mother particle or, in a more general
model, the parent particle may generate a configuration of children
according to some point process.
We refer to \cite{shi2015branching} for an introduction to this very active research topic.

Our goal is to model a population which competes for resources as a
particle system in which particles reproduce, move randomly in space, and
compete with each other locally. We choose here a specific and rather radical form of
interaction: whenever two or more particles are on the same site, they
annihilate. We discuss in Section~\ref{ssec:generalisations} below more general forms of competition. The specific annihilation mechanism makes the system non-attractive in the
sense of interacting particle systems, i.e.\ adding more particles initially can
stochastically decrease the law of the configuration at later times.

A first question about branching
random walks is if the system has a strictly positive survival
probability. In the classical case, that is without annihilation, the answer
is well-known since the number of particles at time $n$ forms a Galton--Watson
process. However, taking into account annihilation, the question is much
more difficult and there are relatively few mathematical papers addressing it, see the
discussion of related literature in Section~\ref{sect:litdiscussion} below.

Assuming that the parameters of the model are such that the survival
probability is indeed strictly positive, the next question is about
invariant measures and the convergence towards the invariant measure in
the case of survival. As for the classical branching random walk or the
contact process, it is clear that the Dirac measure on the empty
configuration is invariant. We can show for our model that in a certain
range of parameters there is complete convergence, i.e. there is exactly
one non-trivial ergodic invariant measure and the law of the process, conditioned
on survival, approaches this invariant measure.

Our model allows for a representation as a probabilistic cellular
automaton. Questions about ergodicity and complete convergence are
notoriously difficult for such systems, we refer to
\cite{mairesse2014around} for an introduction. If we consider the
iteration of the expected number of particles at the sites of the
lattice, we have a deterministic system, a coupled map lattice, see
Section~\ref{sect:CML} below, and especially~\eqref{eqn:cmlattice}. This system is of independent interest and
we expect that it admits, in a certain range of parameters, travelling
wave solutions. Hence our model can be interpreted as a stochastic
perturbation of the coupled map lattice, and this interpretation raises
several interesting questions, which we point out in Section \ref{sec:questions} below.

Let us now give a more precise definition of the model and describe our
results. We study a process $\eta = (\eta_n(x):  x \in \Z^d,  n \ge 0)$
evolving in discrete time on $\Z^d$, where $\eta_n(x)$ denotes the state
of site $x$ at time $n$. We write $\eta_n(x)=1$ if site $x$ is occupied
by exactly one particle at time $n$ and $\eta_n(x)=0$ otherwise. We
denote by $\norm{ \,\cdot\, }$ the sup-norm on $\Z^d$ and define
$B_R(x) = \{ y \in \Z^d : \norm{ y-x} \le R \}$ to be the $d$-dimensional
ball (box) of radius $R \in \N$ centred at $x\in \mathbb Z^d$.
We set $V_R=2R+1$ to be its side length, so that its volume is $V_R^d$.

For fixed $R \in \N$, $\mu >0$, and an initial particle configuration
$\eta_0 \in \{0,1\}^{\Z^d}$, the configurations at later times are
obtained recursively as follows. Given $\eta_n$, $n \ge 0$, for every
$x \in \Z^d$ with $\eta_n(x)=1$ the particle at $z$ dies and gives birth
to a Poisson number of children with mean $\mu$. Each child moves
independently to a uniformly chosen site in $B_R(x)$. Whenever there is
more than one particle at a given site, all the particles at that
site are killed.  This means that if two (or more) children of the same
parent jump to the same site they will disappear, but also children
coming from different parents who jump to the same site will annihilate.
The particles remaining after the annihilation make up the configuration
$\eta_{n+1}$.

The thinning and superposition properties of the Poisson distribution
give the following equivalent description of the model, which is
particularly convenient to carry out calculations.
For a configuration $\eta \in \{ 0,1 \}^{\Z^d}$ and $z \in \Z^d$, define first
the (local) density of particles at $z$ by
\begin{equation}
  \label{eqn:density}
  \delta_R(x; \eta) := V_R^{-d} \sum_{y \in B_R(x)} \eta(y) \in [0,1].
\end{equation}
Fix $\eta_n$ and denote by $N_{n+1}(x)$ the number of
newborn particles at $x$ in the next generation before the annihilation
occurs. This is given by the superposition of the offspring of all particles that
can move to $x$, that is of all $y \in B_R(x)$ with $\eta_n(y)=1$. Thus
$N_{n+1}(x)$ is a Poisson random variable with mean
$\mu \delta_R(x; \eta_n )$. Taking the annihilation into account, it then
holds that
\begin{equation}
  \label{eqn:next_gen_barw_0}
 \eta_{n+1}(x) =
 \begin{cases} 1 \quad \text{if } N_{n+1}(x) = 1, \\
 0 \quad \text{otherwise.}
 \end{cases}
 \end{equation}
Let
\begin{equation}
  \label{eqn:varphi}
  \varphi_{\mu}(w) = \mu w\, e^{-\mu w}, \quad w \in [0,\infty)
\end{equation}
denote the probability that a Poisson random variable with mean $\mu w$
equals 1. By construction, the random variables in the family
$(\eta_{n+1}(x): x \in \Z^d)$ are conditionally independent given $\eta_n$
and by \eqref{eqn:next_gen_barw_0}, \eqref{eqn:varphi} we can represent
our system as
\begin{equation}
  \label{eqn:next_gen_barw}
  \eta_{n+1}(x) =
  \begin{cases} 1 \quad \text{with probability }
    \varphi_{\mu}( \delta_R(x; \eta_n)), \\
    0 \quad \text{otherwise}.
  \end{cases}
\end{equation}
This gives a representation of $\eta$ as a particular example of a
probabilistic cellular automata. We point out that this representation is
only possible because we choose a Poisson offspring distribution. For
more detailed discussion of the assumptions of the model, see the
discussion in Section~\ref{ssec:generalisations} below.

\subsection{Main results}

We can now state the main results of this paper. For the intuition behind
them, we find it useful to first point out a few properties of the function
$\varphi_\mu $ introduced in \eqref{eqn:varphi} which governs the
behaviour of the process:
\begin{enumerate}[(a)]
  \item For $\mu \in (0,1]$, $\varphi_\mu$ has a unique fixpoint
   at $0$, which is attractive.
  \item For $\mu >1$, $\varphi_\mu $ has two fixpoints,
  $0$ and $\theta_\mu = \mu^{-1} \ln \mu$. In this case $0$ is always repulsive.
  \item For $\mu \in (1, e^2)$, $\theta_\mu $ is an attractive fixpoint.
  \item For $\mu > e^2$, there are no attractive fixpoints.
\end{enumerate}

In the case (d), the one point iteration $x\mapsto \varphi_\mu(x)$ has a
rich behaviour. Depending on the value of $\mu $, there can be attractive periodic
orbits or a chaotic behaviour, we refer to \cite[Chapter 9]{thompson2002nonlinear} for discussion about the period doubling and chaotic behaviour of iterated maps.

\begin{figure}[hb]
  \centering
  \begin{tikzpicture}[yscale=6,xscale=2]
    \draw[thick, -stealth] (-0.2, 0) -- (4.1, 0) node[below] {$w$};
    \draw[thick, -stealth] (0, {-0.1/6}) -- (0, 0.45) node[left] {$\varphi_\mu $};
    \draw (0.05, {exp(-1)}) -- (-0.05, {exp(-1)}) node [left] {$1/e$};
    \draw (1,{0.1/6}) -- (1,{-0.1/6}) node[below] {$1$};
    \draw[domain=0:4, smooth, variable=\x, samples=50, thick ] plot (\x, {0.7*\x*exp(-0.7*\x)});
    \draw[domain=0:4, smooth, variable=\x, samples=50  ] plot (\x, {2.0*\x*exp(-2.0*\x)});
    \draw[domain=0:4, smooth, variable=\x, samples=50, dashed ] plot (\x, {8.0*\x*exp(-8.0*\x)});
    \draw[domain=0:0.45, smooth, variable=\x ] plot (\x, \x);
  \end{tikzpicture}
  \caption{Graphs of $\varphi_\mu$ for $\mu = 0.7$ (thick), $2$ and
    $8$ (dashed), together with the identity function.}
\end{figure}
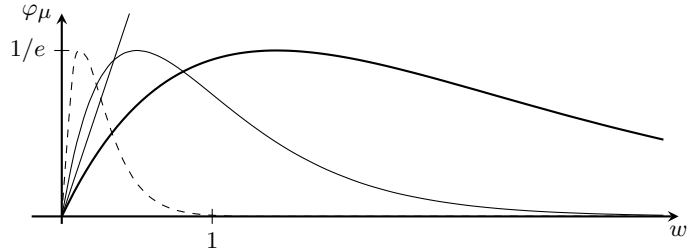

\paragraph*{Extinction}

Our first result identifies a range of parameters $(\mu ,R)$ where the
process dies out a.s. Here, we say that $\eta $ goes extinct \emph{locally} if
$\lim_{n\to\infty} \eta_n(x) = 0$ for every $x\in \mathbb Z^d$, and that
$\eta $ goes extinct \emph{globally}, if $\eta_n \equiv 0$ for all $n$ large
enough.

\begin{theorem}
  \label{thm:R_fixed_extinction}
  For $R \in \N$, let $\mu_1(R), \mu_2(R)$ be the two real solutions of
  \begin{equation*}
    V_R^d \, \varphi_{\mu}\big(V_R^{-d}\big) = 1
  \end{equation*}
  with $1 < \mu_1(R) < \mu_2(R) < \infty$. If $\mu < \mu_1(R)$ or
  $\mu > \mu_2(R)$, then, for all initial conditions $\eta$ goes locally
  extinct a.s.,~and for all initial conditions containing
  only a finite number of particles $\eta $ goes extinct globally
  a.s.~Furthermore,
  $\mu_1(R) \to 1$ and $\mu_2(R) \to + \infty$ as $R \to \infty$.
\end{theorem}

The result of the proposition is not optimal: we expect (based on
  simulations, see Figure~\ref{fig:phase_diagram} in
  Section~\ref{sec:questions} below) that the process goes
extinct for many values $(\mu, R)$ outside of the specified range.  On
the other hand, its proof is relatively simple. It is given in
Section~\ref{sec:extinction} below, and is based on the observation that
for $(\mu ,R)$ in the specified range the killing by annihilation among
siblings is already strong enough to make the expected number of
``surviving'' offspring of a single particle strictly smaller than one, and
thus the branching effectively subcritical, even though $\mu>1$.

\begin{remark}
  The two values $\mu_1(R)$ and $\mu_2(R)$ can be given
  explicitly as
  \begin{equation*}
    \mu_1(R) = -V_R^d \, W_0 ( -V_R^{-d}),
    \quad \mu_2(R) = -V_R^d \, W_{-1} ( -V_R^{-d}),
  \end{equation*}
  where $W_0$ and $W_{-1}$ are the two real branches of the Lambert $W$ function.
  This also describes their asymptotic behaviour as $R\to\infty$, see
  \eqref{eq:mu1asympt} and \eqref{eq:mu2asympt} in the proof of
  Theorem~\ref{thm:R_fixed_extinction} below.
\end{remark}

\paragraph*{Survival}

The second result identifies a range of parameters where it is possible
that the process survives, by which we mean that it survives locally, that is, for every
$x\in \mathbb Z^d$ the set of times $n$ when $\eta_n(x) =1$ is unbounded.
Similarly as in Theorem~\ref{thm:R_fixed_extinction}, the identified
range is not optimal.
  
\begin{theorem}
  \label{thm:survival_eta}
  For every $\mu > 1$ there exists $R_{\mu} \in \N$ such that $\eta$
  survives with positive probability from any non-trivial initial
  condition when $R \ge R_{\mu}$.
\end{theorem}

Inspection of the proof of Theorem \ref{thm:survival_eta} shows that
the following holds.
  
\begin{corollary}\label{cor:survival_eta_cor}
There is $R_0\in \N$ such
  that for every $R \ge R_0$ there exist two values
  $1 < \underline{\mu}_R < \overline{\mu}_R$ such that $\eta$
  survives
  with positive probability
  from any non-trivial initial condition when
  $\mu \in (\underline{\mu}_R, \overline{\mu}_R)$.
  Furthermore, on the event of survival, it holds that
  \begin{equation}
    \label{eq:positivedensity}
    \liminf_{N\to\infty} \frac{1}{N} \sum_{j=1}^N \eta_j(x) > 0
    \quad \text{a.s.\ for any } x \in \Z^d.
  \end{equation}
\end{corollary}

\paragraph*{Ergodicity and complete convergence}

The final set of results discusses the invariant measures of the
process, in the case when the system survives. For this we equip the state
  space $\{0,1\}^{\Z^d}$ with the product topology and the corresponding
  Borel $\sigma $-algebra. In these results we restrict ourselves to
$\mu \in (1,e^2)$, where the non-trivial fixpoint of $\varphi_\mu $ is
attractive, as pointed out above.

\begin{theorem}
  \label{thm:complete_convergence}
  For every $\mu \in (1,e^2)$ there is $R'_\mu < \infty$ such that for
  every $R \ge R'_\mu $ the process $\eta$ has two extremal invariant
  distributions: the first one is trivial and is concentrated on the
  empty configuration $\eta \equiv 0$, and the second one, $\nu_{\mu ,R} $,
  is non-trivial, translation invariant, ergodic, and has exponential
  decay of correlations.

  Furthermore, starting from any non-trivial initial condition the
  process $\eta$, conditioned on non-extinction, converges in
  distribution in the weak topology to the non-trivial extremal
  invariant distribution $\nu_{\mu ,R}$.
\end{theorem}

The driving result behind Theorem~\ref{thm:complete_convergence} is the
following strong coupling property of the system $\eta$, which is of
independent interest.

\begin{theorem}
  \label{thm:coupling}
  Assume that $\mu\in(1,e^2) $ and $R \ge R'_\mu $ satisfy the assumptions of
  Theorem~\ref{thm:complete_convergence}. Then there exists a speed
  $a>0$ (depending on $R, \mu,d$) such that for every pair of (possibly random) initial
  conditions $\eta_0^{(1)}, \eta_0^{(2)} \in \{0,1\}^{\Z^d}$ there exists
  a coupling of the processes $(\eta^{(i)}_n)_{n \in \N_0}$, $i=1,2$,
  with the following property: for each $x \in \Z^d$ there is an
  $\N_0 \cup \{\infty\}$-valued random variable $T_x^{\mathrm{coupl}}$
  (whose exact law will in general depend on the initial conditions and
    on $x$) such that
  $\{ \eta^{(i)}_n \not\equiv 0 \text{ for all } n \in \N, \, i=1,2\}
  \subseteq\{ T_x^{\mathrm{coupl}}<\infty \} $
  a.s.\ and
  \begin{align*}
    \eta^{(1)}_n(y) = \eta^{(2)}_n(y)
    \quad \text{for all } n > T_x^{\mathrm{coupl}} \text{ and }
    \norm{y -x} \le a(n-T_x^{\mathrm{coupl}}).
  \end{align*}
\end{theorem}

\begin{remark}
  From the proof of Theorem~\ref{thm:coupling} it follows that,
  when starting from a finite (or a half-space) initial condition, the
  system $\eta$, given that it survives, will expand into the ``empty
  territory'' at least at some (small) linear speed.  Furthermore,
  simple comparison arguments with supercritical branching random walks
  show that this expansion cannot occur faster than linearly.
  However, identifying an actual linear speed or even an asymptotic
  profile of the expanding population near its tip remains
  a topic for future research.
\end{remark}

\begin{remark}
  Denote by
  $\overline{\delta}_{\mu,R} = \E_{\nu_{\mu,R}}[\eta_0(0)] \in (0,1)$ the
  particle density of the non-trivial invariant measure $\nu_{\mu,R}$
  from Theorem~\ref{thm:complete_convergence}, where $\mu \in (1,e^2)$
  and $R \ge R'_\mu$. By ergodicity, we have almost surely when
  $\eta_0 \sim \nu_{\mu,R}$
  \begin{equation}
    \label{eq:nu.dens}
    \lim_{N \rightarrow \infty} \frac{1}{N} \sum_{n=1}^N \eta_n(x)
    = \lim_{N \rightarrow \infty} \frac{1}{N} \sum_{n=1}^N \delta_R(x; \eta_n)
    = \overline{\delta}_{\mu,R} \quad \text{for every } x \in \Z^d.
  \end{equation}
  By the coupling property from Theorem~\ref{thm:coupling}, Equation~\eqref{eq:nu.dens} holds in fact a.s.\ for any initial condition given
  that the system survives.

  Furthermore, for $1 < \mu < e^2$, inspection of the proof of
  Theorem~\ref{thm:coupling} shows that for every $\varepsilon \in (0,1)$
  there exists $R'_{\mu ,\varepsilon } < \infty$ such that if
  $R \ge R'_{\mu ,\varepsilon }$ then, conditionally on non-extinction,
  \begin{equation}
    \label{eq:nu.dens.blocks}
    \liminf_{N \rightarrow \infty} \frac{1}{N} \sum_{n=1}^N
    \ind_{ \{ \abs{ \delta_R(x; \eta_n)- \theta_{\mu}}  < \varepsilon \} }
    \ge 1- \varepsilon \quad \text{almost surely for every } x \in \Z^d,
  \end{equation}
  where we recall that $\theta_{\mu}$ is the fixpoint of $\varphi_{\mu}$
  (on ``good'' blocks, the particle density is close to $\theta_\mu$,
    see Definition~\ref{def:good_block} below, and good blocks are shown
    to occur with high space-time density).

  Note that \eqref{eq:nu.dens} and \eqref{eq:nu.dens.blocks} together
  imply that $|\overline{\delta}_{\mu,R}-\theta_\mu| \le 2 \varepsilon$
  for $R \ge R'_{\mu ,\varepsilon }$. This corroborates the idea that for
  large $R$, the particle system's behaviour is close to that of the
  corresponding deterministic coupled map lattice, which we discuss in
  Section~\ref{sect:CML} below, and which, as shown in
  Proposition~\ref{prop:cml}, converges to the configuration which is
  constant and equal to~$\theta_\mu$.
\end{remark}
\begin{remark}\label{rmk:generalpca}
If we take
\eqref{eqn:next_gen_barw} as the definition of the model, the
particular form of the function $\varphi_\mu $ used there does not
play a strong role. Our techniques will continue to work, if we
replace $\varphi_\mu $ by another function of a ``similar shape''. In
fact, for survival we only need that $\varphi_\mu : [0,1] \to [0,1]$
is continuously differentiable and strictly positive on $(0,1]$ with
$0$ as an unstable fixpoint. For the proof of the convergence
result, we also need that there is a  unique attracting fixpoint
$\theta_\mu \neq 0$.
\end{remark}

\subsection{Motivations and possible generalisations}
\label{ssec:generalisations}

The construction of our model might seem very rigid. Therefore, we discuss
here the role of the different assumptions and their possible
generalisations.

In the biological literature a wide range of mechanisms of competition has been considered to match experiments on real populations. The terms ``scramble'' and ``contest'', first introduced in \cite{Nicholson1954AnOO}, have become traditional in the literature to distinguish between two different kinds on intraspecific competition. In \emph{contest} competition, an individual is successful if she gets the amount of resources required for survival or reproduction. The logistic branching random walk studied for example in~\cite{birkner2007survival} is a possible mathematical model for a population regulated by contest competition. \emph{Scramble} competition instead involves the exactly equal partitioning of the resources among individuals on each site, such that there is an abrupt change from survival to complete extinction when there is insufficient resource to maintain any individual. Scramble and contest competition are considered two extreme cases and normally some element of both is likely to be observed. However, populations regulated only by scramble competition are still of interest in the biological community, for example in~\cite{brannstrom2005coupled} a version of our model with deterministic branching is studied through a stochastic approximation to a coupled map lattice, in a slightly different setting than ours, but showing similar bifurcation diagrams. 

From a mathematical perspective, if we look at the branching annihilating random walk as a model for a population regulated by scramble competition, we could weaken our hypothesis in several directions. The assumption that the particles jump distribution is uniform over a box
of length $V_R$ is non-essential and is made only for convenience of the
notation. It can in principle be replaced by an arbitrary (centred)
\emph{finite range} transition kernel, and the proofs can be adapted by
suitably replacing the particle density \eqref{eqn:density} by the
convolution of this kernel with $\eta $. The assumption that the number of offspring of a single particle has a
Poisson distribution is more important, as it allows for the essential
representation~\eqref{eqn:next_gen_barw}, and also yields the
conditional independence of $(\eta_{n+1}(x):x\in \mathbb Z^d)$ given
$\eta_n$. Replacing the Poisson distribution would thus require
non-trivial modifications to our proofs.

On the other hand, if we take~\eqref{eqn:next_gen_barw} as a definition of the branching annihilating random walk we obtain an equivalent description of the model as a probabilistic cellular automaton. Cellular automata are extensively considered as models to understand biological self-organisation; for an excellent overview of their wide applicability see for example~\cite{ermentrout1993cellular}. This survey shows how deterministic automata (and probabilistic versions of them), where each state at the next time step is determined solely from earlier states of the cell and its neighbours, have been employed as models for waves in excitable and oscillatory media, predator--prey models and spatial pattern formation. 
In particular, our model is strongly connected with probabilistic cellular automata modelling pattern formation, in which the state of each site is updated by spatially averaging over a neighbourhood and applying a Heaviside function.
Furthermore, a generalisation of this mechanism to weighted averaging has been studied in the context of neural networks as a modification of the Hopfield model for cortical electrical activity, describing the dynamics of parallel (rather than asyncronous) switching from excitatory and inhibitory cells. Going back to our definition~\eqref{eqn:next_gen_barw}, as we pointed out in Remark~\ref{rmk:generalpca} our results still hold if we replace $\varphi_{\mu}$ with a wider class of functions. For further inspiration to future work in connection with probabilistic cellular automata we refer to the rich variety of models surveyed in~\cite{ermentrout1993cellular}.

Finally, both in the interpretation of the branching annihilating random walk as either a regulated population or a probabilistic cellular automaton, the ``hard'' annihilation constraint of at most one particle per site could be relaxed by replacing the definition \eqref{eqn:next_gen_barw_0} of $\eta$ by
$\eta_{n+1}(x) = N_{n+1}(x) \ind_{\{ N_{n+1}(x) \le k \}}$ for some $k \in \N$. Since this modification retains the conditional Poisson and independence properties of the $N_{n}(x)$'s and the sums of truncated Poisson random variables have good concentration properties, we are optimistic that our proofs could be adapted to this scenario with some additional work.

\subsection{Discussion of related results}
\label{sect:litdiscussion}

One of the first models of branching annihilating random walks was
introduced and studied in \cite{bramson1985survival}.
The authors considered a particle system on $\Z$, in which sites can be occupied
as the result of the following mechanisms: particles can either
\textit{jump} to one of the two neighbouring sites at a certain rate or
\textit{branch} into two by giving birth to a new particle on one of the
neighbouring sites. On top of this, particles behave independently except
when they land on a site which is already occupied, in which case both
particles disappear, \textit{annihilate}. The authors show that, starting
from any finite number of particles, the system survives with positive
probability if the jumping rate is small compared to the branching rate
and that the population dies out almost surely if the jumping rate is
sufficiently high. This process is an interacting particle system in the
sense of \cite{liggett1985interacting, liggett1999stochastic}  but it is
not attractive. The authors use contour arguments which rely on the
one-dimensional model they chose.

Very general interacting particle systems on $\Z$ are considered in~\cite{Sudbury2000survival}, where pairwise interactions among neighbours can produce annihilation, birth, coalescence, and exclusion and single individuals can die. Conditions on the rates which ensure positive probability of survival are given by making use of self-duality (which has been proved by the same author in~\cite{Sudbury2000duality}) and supermartingale arguments. 
In~\cite{Bramson1991Annihilating} instead, processes on $\Z^d$ with nearest-neighbour birth at rate 1, annihilation and spontaneous death at rate $\delta$ have been considered. An extinction result for the branching annihilating process started from one particle at the origin is obtained by comparison with the contact process. On the other hand, survival when $\delta$ is small is proved through comparison with oriented percolation. 

In cases where survival can be established, natural questions concern the existence of stationary distributions and weak convergence. 
In~\cite{Sudbury1990branching} a version of the model introduced in~\cite{bramson1985survival} is considered in $\Z^d$ in the case of no random walk, that is, when particles can only move as a consequence of the branching and there is no underlying motion. It is shown that the product measure with rate $1/2$ is the only non-empty limiting distribution. In the case of a double branching and annihilating process on $\Z$ (where each particle can place offspring on both of its neighbouring sites), a richer variety of limiting measures is exhibited. 
In~\cite{Bramson1991Annihilating} the authors prove that when $\delta=0$ the product measure with density $1/2$ is stationary and is the limiting measure, thus obtaining independently the same result proved in~\cite{Sudbury1990branching}. Furthermore it is shown that for any $\delta$ there are at most two extremal translation invariant stationary distributions, and if $\delta$ is small there exists a non-trivial stationary distribution. 

Another question of interest is whether branching processes with annihilation satisfy duality relations. In~\cite{Athreya2012systems} processes in continuous time are considered, in which particles can annihilate, branch, coalesce or die. The authors show that annihilation does not play a key role in a duality relation: the process with annihilation is dual to a system of interacting Wright-Fisher diffusions, and this result holds also if annihilation is suppressed (but in the case of annihilation the duality function is different and more complicated). It would be highly interesting to find a useful duality relation for our model as well.

Versions of branching annihilating processes in discrete time are generally more difficult to deal with, since continuous time implies that changes in the configuration can only occur one site at a time, sequentially as opposed to in parallel. A discrete-time analogous of~\cite{bramson1985survival} has been considered in~\cite{Alili2001surviving} for a model on $\Z$, where particles at each time move with probability $1-\varepsilon$ or branch with probability $\varepsilon$, with the rule that two particles occupying the same site will annihilate. The authors show that, if the branching probability is small enough, for any finite initial configuration of particles the probability $p(t)$ that at least one site is occupied at time $t$ decays exponentially fast in $t$. 

A branching
annihilating random walk on the complete graph is studied in~\cite{perl2015extinction}. This process evolves in discrete
time, the number of offspring is Poisson distributed with mean $\mu$
and each one of them independently moves to one of the neighbouring sites
of their parent. This corresponds to our model on the complete graph. Since on a finite graph there is always a positive
probability of total annihilation in one step, the system eventually dies
out at some finite time. The authors prove that if $\mu >1$ the process on the
complete graph with $N$ vertices has an exponentially long lifetime in $N$
and that its last excursion from the ``equilibrium value'' $\theta_\mu N$ before it reaches
the zero state is logarithmic in $N$.

Besides systems where particles can annihilate, recent research
directions have also been focusing on spatial branching systems in which
the interaction among particles is regulated by a competition kernel
which can reduce the average reproductive success of an individual at a
given site. In this case, rather than annihilating particles in areas
with high particle density, the existing particles will produce fewer
offspring. Spatial models with local competition are for example
investigated in \cite{
Etheridge2004, 
birkner2007survival, 
blath2007coexistence,
Kondratievuniverse, 
Mueller2015,
maillard2021branching}. 
The two papers most related to our present work
are \cite{birkner2007survival, maillard2021branching}.

In~\cite{birkner2007survival} a discrete
time branching system with a finite range (and thus \textit{local})
competition kernel is considered. The authors show that the system survives with positive
probability if the competition term is small enough and obtain complete
convergence of the system to a non-trivial equilibrium for some choices
of the model parameters. The strategy used in \cite{birkner2007survival}
to prove survival is building a comparison with an oriented percolation
model. We will use similar ideas to show survival for our branching
annihilating random walk, as well as complete convergence.

A more recent work~\cite{maillard2021branching} considers a process in continuous time and
non-local competition kernels, where the range of interaction can be
arbitrary, even infinite. Using a contour argument, the authors prove that in
the low competition regime the system survives globally. In the same
regime, they also provide a shape theorem, showing that the asymptotic
spreading speed of the population is the same as in the branching random
walk without competition.

Since we work in discrete time, our model is not an interacting particle
system in the sense of \cite{liggett1985interacting,
  liggett1999stochastic} but rather a probabilistic cellular automaton, as discussed above.
We refer to \cite{mairesse2014around} for a survey on probabilistic
cellular automata. Ergodicity and complete convergence for probabilistic
cellular automata is a notoriously difficult topic where many
proof techniques are model-dependent. For attractive systems there are
still some general tools as monotonicity and subadditivity, see
\cite{hammersley1974postulates}. We refer to \cite{FernandezLouisNardi}
for a collection of recent results.

\subsection{Auxiliary coupled map lattice}
\label{sect:CML}

Our work also raises questions about coupled map lattices which are
deterministic versions of probabilistic cellular automata, see
\eqref{eqn:next_gen_barw}, and which, in our examinations of the BARW,
serve as an intuitional guide for the proofs of the survival and the
complete convergence. This coupled map lattice is a deterministic
$[0,e^{-1}]^{\mathbb Z^d}$-valued process  $\Xi_n$ (note that
  $\max_{w\ge 0} \varphi_\mu (w) = e^{-1}$) defined, given any initial
condition $\Xi_0$, by the iteration of
\begin{equation}
  \label{eqn:cmlattice}
  \Xi_{n+1}(x) := \varphi_\mu (\delta_R (x; \Xi_n)).
\end{equation}
At least for $R$ large, locally, the dynamics of this process is a good
approximation for the dynamics of the ``density profile''
$\delta_R (\cdot; \eta_n)$ of $\eta $, as can be heuristically seen from
\eqref{eqn:next_gen_barw} and the law of large numbers.

We will prove and exploit the fact that in the regime when $\varphi_\mu $
has the unique attractive fixpoint $\theta_\mu $, that is for
$\mu \in (1,e^2)$, when starting from a non-zero initial condition, the
coupled map lattice converges locally to $\theta_\mu $, and the region where it is
close to this value expands.

\begin{proposition}
  \label{prop:cml}
  Let $\mu \in (1, e^2)$ and assume that $\Xi_0(0)>0$. Then
  \begin{equation*}
    \lim_{n\to\infty} \Xi_n(z) = \theta_{\mu}\qquad \text{for all }
    z\in \mathbb Z^d,
  \end{equation*}
  and for every $\varepsilon >0$ there is a speed
  $a = a(\mu,\varepsilon,\Xi_0(0)  )>0$ such that
  $\Xi_n(x)\in (\theta_\mu-\varepsilon ,\theta_\mu+\varepsilon )$ for all
  $\abs x \le a n$.
\end{proposition}

We believe that for localised or half-space initial conditions, the
process $\Xi $ will approach a ``travelling wave''. While there is a rich
literature addressing travelling waves, we were not able to find results
which literally apply in our context, in particular since our model has
discrete time and space. We thus prove the above (weaker and non-optimal)
proposition by rather bare hand arguments, which involve a construction
of a ``travelling wave sub-solution'', see Section~\ref{ss:profiles}
below. Travelling waves in the context of PDEs have been widely studied,
also with a view of  biological applications. The existence of travelling waves
has also been considered quite extensively in the context of discrete
time, \emph{continuous space} models, see
e.g.~\cite{weinberger1978asymptotic, LiLewisWeimberger2009, Kot92,
  KotSchaffer1986}. In particular, existence of such travelling waves has been shown in situations where $\varphi_\mu$ in
\eqref{eqn:cmlattice} is replaced by an increasing (and hence monotone)
function \cite{hammersley1974postulates,weinberger1978asymptotic}.

The regime $\mu > e^2$ is also very interesting. In this case the
iteration of $\varphi_\mu $ does not converge to a single point but to a
stable orbit, which as $\mu$ increases beyond $e^2$  will have an increasing number of elements. We refer to~\cite[Chapter 9]{thompson2002nonlinear} for a description of bifurcation diagrams and chaos for one-dimensional maps.
In this regime, we are not aware of results in the
literature covering the coupled map lattice model. But even given such
results, the behaviour of the stochastic system might be different and
more difficult to control than in the stable-fixpoint case treated here.
We leave these questions for future work.

\section{Preliminary results and tools}
\label{sec:preliminary}

In this section we collect some preliminary results that will be used
throughout the paper.

\subsection{A general coupling construction}
\label{ss:coupling}

We will frequently make use of the following construction allowing to define the
process $\eta $ for all initial conditions simultaneously and also
allowing to compare $\eta $ with other particle systems, in particular with
monotone ones.

Let $U(x,n)$, $x\in \mathbb Z^d$, $n\in  \mathbb N_0$, be a
collection of i.i.d.~uniform random variables on $[0,1]$. Recall the
definition of the function $\varphi_\mu $ from \eqref{eqn:varphi}, and
let $\psi : [0,1] \to \R_+$ be any non-decreasing function satisfying
\begin{equation}
  \label{eqn:psiphirelation}
  \psi(w) \le \varphi_\mu(w) \qquad
  \text{for all }w \in [0, 1] \cap V_R^{-d} \Z,
\end{equation}
that is, for all possible values of the density $\delta_R(\cdot;\eta_n )$.
Then, for any initial conditions $\eta_0$, $\widetilde \eta_0$,
define, recursively for $n\ge 0$,
\begin{align}
  \label{eqn:etacouplconstr}
  \eta_{n+1}(x) &= \ind_{\{ U(x,n+1)
      \le \varphi_{\mu}(\delta_R(x;\eta_n))\}},
  \\
  \label{eqn:etatilde}
  \widetilde{\eta}_{n+1}(x) &= \ind_{\{ U(x,n+1)
      \le \psi(\delta_R(x;\widetilde{\eta}_n))\}}.
\end{align}
The construction \eqref{eqn:etacouplconstr} of $\eta $ is morally the
analogue of the common graphical construction of an interacting particle
system in our context, and can be viewed as a stochastic flow on the
configuration space $\{0,1\}^{\Z^d}$. The next lemma gives its main
properties.

\begin{lemma}[General coupling construction]
  \label{lem:coupling_lemma}
  \begin{enumerate}[(a)]
    \item The process $\eta$ defined by \eqref{eqn:etacouplconstr} has
    the law of the branching-annihilating random walk with parameters
    $\mu $ and $R$ and initial condition $\eta_0$.

    \item If $\widetilde{\eta}_0(x) \le \eta_0(x)$ for all
    $x\in \mathbb Z^d$, then
    $\widetilde{\eta}_n(x) \le \eta_n(x)$ for all $n \in \N$ and
    $x\in \mathbb Z^d$.
  \end{enumerate}
\end{lemma}

\begin{proof}
  Part (a) follows immediately from \eqref{eqn:next_gen_barw}. To see
  part (b) assume $\widetilde\eta_{1}(x)=1$ for some $x \in \Z^d$. Then,
  by construction  $U(x,1) \le \psi( \delta_R(x; \widetilde\eta_0))$.
  Since $\widetilde\eta_0\le\eta_0$ and $\psi $ is non-decreasing, and
  $\varphi_{\mu}$ dominates $\psi$, this yields
  $U(x,1) \le \varphi_{\mu}(\delta_R(x;\eta_0))$, and so $\eta_{1}(x)=1$.
  It follows that $\widetilde\eta_1 \le \eta_1$, and by iteration,
  $\widetilde\eta_n \le \eta_n$.
\end{proof}

In what follows we always assume that $\eta $ is constructed as in
\eqref{eqn:etacouplconstr} and define the filtration
\begin{equation}
  \label{eqn:filtration_Fn}
  \mathcal F_n := \sigma\big( \eta_0(x) : x\in \mathbb Z^d \big)
  \vee \sigma \big(U(x,i): x\in \mathbb Z^d, i\le n\big)
  \supseteq \sigma
  \big(\eta_i(x): x\in \mathbb Z^d, i\le n\big).
\end{equation}

\subsection{Concentration and comparison with deterministic profiles}
\label{ss:conc+comp}

As remarked under \eqref{eqn:varphi}, the random variables
$(\eta_{n+1}(x): x \in \mathbb Z^d)$ are conditionally independent given
$\eta_n$. Therefore, the density $\delta_R(x;\eta_{n+1})$ should
concentrate, at least for $R$ large. We need estimates providing
quantitative control of this concentration. These estimates involve
certain sequences of functions $\xi_k^{\pm}$ on $\mathbb Z^d$, which we
call \emph{comparison density profiles}, that have the property that if
at some time $t$ the local density of $\eta_t$ is controlled by
$\xi_k^{\pm}$, then, at least locally, the density of $\eta_{t+1}$ is
controlled by $\xi^{\pm}_{k+1}$. In fact, the sequences  $\xi_k^-$ and
$\xi_k^+$ that we use later can be regarded as a travelling wave sub- and
super-solution, respectively, of the coupled map lattice
iteration~\eqref{eqn:cmlattice}.

\begin{definition}
  \label{def:density_profiles}
  For a given $\varepsilon, \delta  >0$, \emph{comparison density
    profiles} are deterministic functions
  $\xi_k^{-}, \xi_k^{+} : \Z^d \to [0,\infty)$, $k=0,1,\dots,k_0$,
  satisfying:
  \begin{enumerate}[(i)]
    \item \label{profile:i}
    For every $k=0,\dots,k_0$,
    $\xi_k^{-}(\cdot) \le \xi_k^{+}(\cdot)$.

    \item \label{profile:ii}
    For every $k=0,\dots,k_0$,
    $\Supp(\xi_k^-):=\{x\in \mathbb Z^d: \xi_k^-(x) > 0\}$ is finite, and
    $\xi_k^-(x)\ge \varepsilon $ for
    every $x\in \Supp(\xi_k^-)$.

    \item \label{profile:iii}
    For every $k = 0,\dots,k_0-1$, and $x\in \Supp(\xi_k^-)$ it holds
    that if $\zeta : B_R(x) \to \mathbb R$ satisfies $\zeta (y) \in
    [\xi_k^{-}(y), \xi_k^{+}(y)]$ for all $y\in B_R(x)$, then
    \begin{equation}
      \label{eqn:xi.comp.bd1}
      (1+\delta) \xi_{k+1}^{-}(x)
      \le V_R^{-d} \sum_{y \in B_R(x)} \varphi_\mu ( \zeta(y))
      \le (1-\delta) \xi_{k+1}^{+}(x).
    \end{equation}
  \end{enumerate}
\end{definition}

Note that $\xi_k^{-}, \xi_k^{+}$ will in general depend on $R$, $\mu$,
$\varepsilon $ and $\delta $, but we do not make this explicit in the
notation (in fact, $\delta, \varepsilon$ could also depend on $R$ and
  $\mu$).

\begin{lemma}
  \phantomsection
  \label{lem:dens.dev}
  \begin{enumerate}[(a)]
    \item
    For comparison density profiles $\xi^{\pm}_{k}$, if for some $x\in \mathbb Z^d$ and
    $k \in \{ 0,\dots,k_0-1 \}$
    \begin{equation}
      \label{eqn:etak.dens.ass}
      \delta_R(y; \eta_k) \in \big[\xi_k^{-}(y), \xi_k^{+}(y)\big]
      \quad \text{for all } y \in B_R(x),
    \end{equation}
    then
    \begin{equation}
      \label{eqn:profile_propagation}
      \P\Big( \xi_{k+1}^{-}(x) \le \delta_R(x; \eta_{k+1}) \le \xi_{k+1}^{+}(x)
        \,\Big|\, \mathcal F_k \Big)
      \ge 1-2 \exp(- c V_R^d),
    \end{equation}
    where
    $c = ({\delta \varepsilon})/\big({1/(2\delta \varepsilon) + 2/3}\big)$.

    \item If, in \eqref{eqn:xi.comp.bd1}, $\varphi_\mu $ is replaced by
    any $\psi$ satisfying \eqref{eqn:psiphirelation}, then statement (a)
    holds for the monotone dynamics
    $\widetilde\eta$ defined in \eqref{eqn:etatilde} in place of $\eta $.
  \end{enumerate}
\end{lemma}

\begin{remark}
  \label{rem:onlyLB}
  If only a lower bound is required, as e.g.~in the proof of survival,
  one can use the ``trivial'' upper bound for $\xi_n^{+}$, namely
  $\xi_n^{+}(\cdot) \equiv \max(\varphi_\mu)/(1-\delta) = e^{-1}/(1-\delta)$,
  and then apply \eqref{eqn:profile_propagation} only for the lower bound.
\end{remark}

\begin{proof}[Proof of Lemma~\ref{lem:dens.dev}]
  We only show (a), the proof of (b) is completely analogous.
  We consider first the lower bound, that is we want to show that the conditional
  probability of the event $\{{\delta_R(x;\eta_{k+1}) < \xi_{k+1}^-(x)}\}$ is small, given $\mathcal{F}_k$.
  Note that, by \eqref{eqn:etak.dens.ass} and \eqref{eqn:xi.comp.bd1},
  \begin{equation*}
    \sum_{y \in B_R(x)} \E[ \eta_{k+1}(y) \mid \mathcal F_k ]
    = \sum_{y \in B_R(x)} \varphi_\mu\big(\delta_R(y; \eta_k)\big) \ge (1+\delta) V_R^d
    \xi_{k+1}^{-}(x).
  \end{equation*}
  Therefore,
  \begin{equation}
    \label{eqn:dens.dev.1pt.lower.1}
    \begin{split}
      & \P\Big( \delta_R(x; \eta_{k+1}) <  \xi_{k+1}^{-}(x) \,\Big|\,
        \mathcal F_k \Big)
      \\ & \quad \le
      \P\Big( {\textstyle \sum_{y \in B_R(x)}} \big( \eta_{k+1}(y)
          - \E[ \eta_{k+1}(y) \,|\, \mathcal F_k ] \big)
        < - \delta V_R^d \xi_{k+1}^{-}(x) \,\Big|\, \mathcal F_k \Big)
    \end{split}
  \end{equation}
  and
  \begin{align*}
    & \Var\big( \delta_R(x;\eta_{k+1}) \,\big|\, \mathcal F_k \big)
    = V_R^{-2d}\sum_{y \in B_R(x)} \varphi_\mu\big( \delta_R(y; \eta_k) \big)
    \big( 1- \varphi_\mu\big( \delta_R(y; \eta_k) \big) \big)
    \le \frac{1}{4} V_R^{-d}.
  \end{align*}

  We now apply the Bernstein inequality (which we recall in
    Lemma~\ref{lem:Bernstein} in the \hyperref[appendix]{Appendix}) to the right-hand side of
  \eqref{eqn:dens.dev.1pt.lower.1} with $n=V_R^d$,
  $\sigma_n \le V_R^{-d/2}/2$, $m_n \le 1$ and
  $w = \delta V_R^d \xi_{k+1}^{-}(x) \ge \delta \varepsilon V_R^d$
  (since, by assumption \eqref{profile:ii} $\xi_{k+1}^-(x)\ge \varepsilon$ if
    $\xi_{k+1}^{-}(x) > 0$, and there is nothing to prove if
    $\xi_{k+1}^{-}(x) = 0$). The expression in the exponent of the
  right-hand side of \eqref{eqn:Bernsteinineq} then satisfies
  \begin{equation*}
    \frac{w^2}{2\sigma_n^2 + (2/3) m_n w} = \frac{w}{2\sigma_n^2/w + (2/3) m_n}
     \ge \frac{w}{V_R^d/(2 w) + 2/3}
    \ge \frac{\delta \varepsilon}{1/(2\delta \varepsilon) + 2/3} V_R^d,
  \end{equation*}
  which completes the proof of the lower bound in
  \eqref{eqn:profile_propagation}.

  The proof of the upper bound, that is showing that
  the probability (conditional on $\eta_k$) of the event
  $\{ \delta_R(x; \eta_{k+1}) > \xi_{k+1}^{+}(x)\}$
  is small, is completely analogous.
\end{proof}

\subsection{Lower bounds on travelling waves}
\label{ss:profiles}

The goal of this section is to construct explicit comparison density profiles
$\xi_k^-$ which can later be used as the lower bounds on
$\delta_R(\cdot; \eta)$ in the proofs of survival and complete
convergence. As pointed out before, these can be viewed as travelling
wave sub-solutions to the iteration~\eqref{eqn:cmlattice}.

We start by providing the basic building block for this construction. To
this end we concentrate first on the one-dimensional setting. For
parameters $a>1$, $\varepsilon_0 \in (0,1)$, $w >0, s> 0$ and
$R\in \mathbb N$ we say that a non-decreasing function
$f : \Z \to [0, \infty)$ is a linear travelling wave shape with width
$\lceil w R \rceil$, shift $\lceil s R \rceil$, growth factor $a$ and
minimal step size $\varepsilon_0$ if it fulfils
\begin{equation}
  \label{eqn:f_properties}
   f(x) = 0 \, \text{ for } \, x < 0, \qquad f(0)=\varepsilon_0,
  \qquad f(x) = 1 \, \text{ for } \, x \ge \lceil w R \rceil
\end{equation}
and
\begin{equation}
  \label{eqn:density_domination}
   a\delta_R(x; f) \ge f\big(x+\lceil s R \rceil\big) \quad
  \text{ for all } x \in \Z .
\end{equation}
In this parametrisation, we think of a ``wave profile'' which, when
subjected to one iteration of the operation
$f(\cdot) \mapsto a \delta_R(\cdot;f)$,  moves to the left by at least
$\lceil s R \rceil$ in each time step. Note that by construction, one
necessarily has that $s \le 1$.

We now show that such a function $f$ exists for any $a>1$ and $R$ large.

\begin{lemma}
  \label{lem:f1dprofilelbd}
  For every $a>1$, there is $w\ge 2$, $\varepsilon_0\in (0,1)$,
  $s\in(0,1)$, and $R_0\in\mathbb N$ such that the function
  \begin{equation*}
    f(x) = \min\big\{ (\varepsilon_0 +  x/\lceil wR \rceil)
      \ind_{x \ge 0}, 1 \big\}
  \end{equation*}
  satisfies \eqref{eqn:f_properties}, \eqref{eqn:density_domination} for
  all $R \ge R_0$.
\end{lemma}

The proof of Lemma~\ref{lem:f1dprofilelbd} is a straightforward, albeit
somewhat lengthy computation, and is given in
Section~\ref{sec:pf.lem:f1dprofilelbd}.  In
fact, with even lengthier computations it could be shown that the lemma
holds for any $R \ge 1$.

Using this travelling wave shape we can now define the desired comparison
density profile $\xi^-_n$. For this, fix $R_{\mathrm{init}} \in \N$ with
$R_{\mathrm{init}} > 2R $ and set, for $x\in \mathbb Z$,
\begin{equation}
  \label{eqn:xitilde}
  \widetilde\xi_n(x) := f\big(R_{\mathrm{init}}
    + n \lceil s R \rceil+ \lceil w R \rceil -|x|\big)
\end{equation}
with $f$ from Lemma~\ref{lem:f1dprofilelbd}, see
Figure~\ref{fig:xi_profiles} for an illustration. Note that by
construction, ${\widetilde\xi_n(\cdot) \equiv 1}$ on
$B_{R_{\mathrm{init}} + n \lceil s R \rceil}(0)$ and
$\Supp(\widetilde\xi_n) = B_{R_{\mathrm{init}}
  + n \lceil s R \rceil + \lceil w R \rceil}(0)$.
Furthermore, using \eqref{eqn:density_domination},
$a \delta_R(x; \widetilde\xi_n) \ge \widetilde\xi_{n+1}(x)$ for all
$x \in \Z$, and $\widetilde\xi_n(x)>0$ implies
$\widetilde\xi_n(x) \ge \varepsilon_0$.

\begin{figure}[t]
  \centering
  \begin{tikzpicture}[scale = 0.75]
    \node[right] at (8.31,0){$\mathbb{Z}$};

    \foreach \j in {-83,...,83}
    \draw (\j/10,0)--+(-90:0.05);

    \draw (0,0.05)--+(-90:0.15);
    \node[below] at (0,-0.1){$0$};

    \draw (-4,0.05)--+(-90:0.1);
    \draw (-5,0.05)--+(-90:0.1);
    \draw (-8,0.05)--+(-90:0.1);

    \draw (4,0.05)--+(-90:0.1);
    \draw (5,0.05)--+(-90:0.1);
    \draw (8,0.05)--+(-90:0.1);

    \foreach \j in {-20,...,20}
    \filldraw[blue]  (\j/10,2) circle (0.5pt);

    \foreach \j in {-30,...,-20}
    \filldraw[red] (\j/10,2) circle (0.5pt);

    \foreach \j in {20,...,30}
    \filldraw[red] (\j/10,2) circle (0.5pt);

    \foreach \j in {50,...,60}
    \filldraw[black] (\j/10,2) circle (0.5pt);
    \foreach \j in {-50,...,-60}
    \filldraw[black] (\j/10,2) circle (0.5pt);
    \foreach \j in {-30,...,-40}
    \filldraw[black] (\j/10,2) circle (0.5pt);
    \foreach \j in {30,...,40}
    \filldraw[black] (\j/10,2) circle (0.5pt);

    \foreach \j in {-1,0,1}
    {
      \filldraw[black] (4.5+\j*0.1,2) circle (0.1pt);
      \filldraw[black] (-4.5+\j*0.1,2) circle (0.1pt);
    }

    \foreach \j in {-39,...,-20}
    {
      \filldraw[blue] (\j/10, \j/10+4) circle (0.5pt);
      \filldraw[red] (-1+\j/10, \j/10+4) circle (0.5pt);
      \filldraw[black] (-4+\j/10, \j/10+4) circle (0.5pt);
    }

    \foreach \j in {20,...,39}
    {
      \filldraw[blue] (\j/10, 4-\j/10) circle (0.5pt);
      \filldraw[red] (1+\j/10, 4-\j/10) circle (0.5pt);
      \filldraw[black] (4+\j/10, 4-\j/10) circle (0.5pt);
    }

    \draw[stealth-stealth] (2,-0.7)--+(0:1.99);
    \draw[stealth-stealth]  (4.01,-0.7)--+(0:0.98);
    \draw[stealth-stealth]  (0,-0.7)--+(0:1.98);

    \node[below] at (3,-0.65){$\lceil wR \rceil$};
    \node[below] at (4.5,-0.65){$\lceil sR \rceil$};
    \node[below] at (1,-0.65){$R_{\mathrm{init}}$};

    \node at (3,1.5){{\color{blue}{ $\xi_0^-$}}};
    \node at (4.2,1.5){{\color{red}{ $\xi_{1}^-$}}};
    \node at (7.2,1.5){{\color{black}{ $\xi_{n}^-$}}};
  \end{tikzpicture}
  \caption{The one-dimensional deterministic comparison density profile
    $\xi_n^-$ built from the linear travelling wave shape $f$, with fronts
    of width $\lceil wR \rceil$ that get shifted outward by
    $\lceil sR \rceil$ in every time step. }
  \label{fig:xi_profiles}
\end{figure}
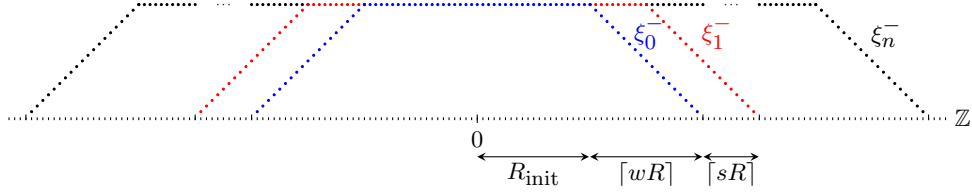

Finally, for any $d \ge 1$, write $x=(x_1,\dots,x_d)$ and set
\begin{equation}
  \label{eqn:xi_n}
  \xi^-_n(x) := b \prod_{i=1}^d \widetilde{\xi}_n(x_i), \quad
  x \in \Z^d, n \in \N_0
\end{equation}
with some $0 < b \le 1$ that will be suitably tuned later. Note that
$\xi^-_n$ implicitly depends on $d$, $R$, $R_{\mathrm{init}}$, $a$ and $b$
but our notation does not make this explicit. We summarise the relevant
properties of $\xi^-_n$ in the following lemma.

\begin{lemma}
  \label{lem:xi_n}
  The functions $\xi_n^-$ have the following properties:
  \begin{enumerate}[(i)]
    \item $0 \le \xi^-_n(x) \le b$ for every $n\in \mathbb N_0$ and
    $x\in \mathbb Z^d$,

    \item for every $n\in \mathbb N_0$, $\xi^-_n(\cdot) \equiv b$ on
    $B_{R_{\mathrm{init}} + n \lceil s R \rceil}(0)$ and
    $\Supp(\xi^-_n) = B_{R_{\mathrm{init}} + n \lceil s R \rceil + \lceil w R \rceil}(0)$,

    \item $a^d \delta_R(x; \xi^-_n) \ge \xi^-_{n+1}(x)$ for all $n \in \N_0$,
    $x \in \Z^d$,
    \item $\xi^-_n(x)>0$ implies $\xi^-_n(x) \ge b \varepsilon_0^d$.
  \end{enumerate}
\end{lemma}

\begin{proof}
  The properties (i), (ii) and (iv) follow directly from
  \eqref{eqn:f_properties}, \eqref{eqn:xitilde} and \eqref{eqn:xi_n}.
  Using \eqref{eqn:density_domination}, \eqref{eqn:xitilde},
  \eqref{eqn:xi_n}, it follows moreover that
  \begin{align*}
    a^d \delta_R(x; \xi^-_n)
    & = a^d V_R^{-d} \sum_{y \in B_R(0)} \xi^-_n(y+x)
    \\& =  b a^d V_R^{-d}  \sum_{y_1=-R}^R \dots \sum_{y_d=-R}^R
    \prod_{i=1}^d \widetilde{\xi}_n(x_i + y_i)
    \\& = b \prod_{i=1}^d \bigg( a V_R^{-1}\sum_{y= -R}^R
      \widetilde{\xi}_n(x_i + y)  \bigg)
    = b \prod_{i=1}^d \big(a \delta_R(x_i; \widetilde{\xi}_n )\big)
    \\& \ge b \prod_{i=1}^d \widetilde{\xi}_{n+1}(x_i)
    = \xi^-_{n+1}(x),
  \end{align*}
  which shows (iii) and completes the proof.
\end{proof}

\section{Survival for large \texorpdfstring{$R$}{R}:
  Proof of Theorem~\ref{thm:survival_eta} and Corollary~\ref{cor:survival_eta_cor}}
\label{sec:survival}

In this section we prove Theorem~\ref{thm:survival_eta}, stating that the
system survives for any $\mu >1$, given that $R$ is chosen sufficiently
large; we also prove Corollary~\ref{cor:survival_eta_cor}. The proof is based on the comparison with a monotone system
$\widetilde \eta$, which in turn is shown to survive using a comparison
with finite range oriented percolation. The latter is a by now classical
technique for interacting particle systems, we refer to
\cite{Czuppon2016thesis}, \cite{lanchier2017stochastic} or
\cite{swart2017course} for recent and reader-friendly introductions.

The monotone system $\widetilde \eta $ is constructed as in
Section~\ref{ss:coupling}: we first fix parameters
$\widetilde a \in (1, \mu)$ and $b\in (0,1)$, so that the function $\psi$
defined by
\begin{equation}
  \label{eqn:psi}
  \psi(w) := \widetilde a (w \wedge b)
\end{equation}
satisfies \eqref{eqn:psiphirelation}. This is possible since $\mu >1$.
With this $\psi$, we define $\widetilde \eta $ as in \eqref{eqn:etatilde}
and simultaneously $\eta$ as in \eqref{eqn:etacouplconstr} on the
probability space supporting the i.i.d.~uniform random variables
$(U(x,n))_{x\in \mathbb Z^d, n\in \mathbb N_0}$.

We then fix  $a>1$ such that $a^d < \widetilde a$, and for this choice of
$a$, we fix $R_0$, $w$, $s$ and $\varepsilon_0$ according to
Lemma~\ref{lem:f1dprofilelbd}. For $R \ge R_0$, we set
$R_{\mathrm{init}} := \lceil w R /2\rceil$ and define $\xi^-_n$ as in
\eqref{eqn:xi_n}. We claim that $\xi^{-}_n(x)$ (and the trivial $\xi^+_n$,
  as explained in Remark~\ref{rem:onlyLB}) is a comparison density profile in the
sense of Definition~\ref{def:density_profiles} with
$\delta = (\widetilde{a}/a^d) - 1$ and $\varepsilon = b \varepsilon_0^d$.
Moreover the lower bound of \eqref{eqn:xi.comp.bd1} even holds with $\psi$
in place of $\varphi_\mu $. Indeed, \eqref{profile:i} is trivially true,
\eqref{profile:ii} follows from Lemma~\ref{lem:xi_n}(iv). To show
\eqref{profile:iii}, that is \eqref{eqn:xi.comp.bd1} (with $\psi$ in
  place of $\varphi_\mu $), let $\zeta=(\zeta(y)) \in [0,1]^{\Z^d}$ be
such that $\zeta(\cdot) \ge \xi_n^-(\cdot)$ for some $n \in \N_0$. Then,
using Lemma~\ref{lem:xi_n}(iii) for the last inequality,
\begin{align*}
  V_R^{-d} \sum_{y \in B_R(x)} \psi \big( \zeta(y) \big)
  & = V_R^{-d} \sum_{y \in B_R(x)}
  \widetilde{a}\big(\zeta (y)  \wedge b\big)  \\
  & \ge \frac{\widetilde{a}}{a^d}\cdot a^d V_R^{-d} \sum_{y \in B_R(x)}
  \xi^-_n(y)
  \\&= \frac{\widetilde{a}}{a^d}\cdot a^d \delta_R(x; \xi_n^-)
   \ge \frac{\widetilde{a}}{a^d}\cdot \xi_{n+1}^-(x) ,
\end{align*}
as required. As a consequence, we will later be able to apply the
concentration result of Lemma~\ref{lem:dens.dev}(b) to the process
$\tilde \eta$.

Define $R'_{\block} = \lceil wR /2\rceil$. To set up the comparison with
oriented percolation, we coarse-grain the system by using blocks spaced
by $L'_{\block} := 2R'_\block$, of side length
$L_{\block} := 5 L'_{\block}$ and temporal size
$T_\block := \big\lceil \lceil w R \rceil/\lceil s R \rceil \big\rceil$.
Since we often refer to radii rather than block lengths, it is convenient
to define $R_{\block} = L_\block/2$.

For $(z,t)$ in the sub-lattice
$\mathbb L:= L'_\block \mathbb Z^d \times T_\block \mathbb N_0$, we define
\begin{equation*}
  \Block(z,t) := \big\{ (x,n) \in \Z^d \times \N_0
    :\norm{x - z} \le R_\block, \, t \le n \le t+ T_\block \big\}.
\end{equation*}
Note that blocks in the same time-layer have non-trivial overlap with
their neighbours but the number of overlapping neighbours in $\mathbb{L}$
per block does not grow with $R$. In the time direction, only the top
time slice of a given block coincides with the bottom layer of the next
block(s).

\begin{definition}
  \label{def:well'started'survival}
  We call $\Block(z,t)$ \emph{well-started} if the density of the
  monotone system $\widetilde \eta $ dominates the (suitably shifted)
  density profile $\xi_0^-$ at the bottom of the block, that is
  \begin{align}
    \label{eqn:wellstarted}
    \delta_R\big(x; \widetilde\eta_{t} \big)
    \ge \xi_0^{-}(x-z)
    \quad \text{for } \norm{x-z}\le R_\block.
  \end{align}
\end{definition}

Note that for any $(z, t) \in \mathbb L$ the event
$\{ \Block(z,t) \text{ is well-started} \} $ is measurable with respect
to the filtration $\mathcal{F}_t$, which was defined in
\eqref{eqn:filtration_Fn}.

\begin{definition}
  \label{def:good_block_survival}
  $\Block(z,t)$ is called \emph{good} if it is well-started and
  the random variables $U(x,n)$
  are such that the domination property of \eqref{eqn:wellstarted} propagates over the block. That is,
  $\Block(z,t)$ is \emph{good} if
  it holds that
  \begin{align*}
    \delta_R\big(x; \widetilde\eta_{t+n} \big)
    \ge \xi_n^{-}(x-z)
    \quad \text{for } \norm{x-z}\le R_\block, \, n=0,\dots, T_\block.
  \end{align*}
\end{definition}

The properties of the comparison density profiles $\xi_n^{-}$, see
Lemma~\ref{lem:xi_n}, enforce
\begin{align}
  \label{goodmakeswell}
  \{ \Block(z,t) \text{ is good} \} \; \subseteq
  \bigcap_{\substack{z' \in L'_\block \Z^d\, : \\ \norm{z-z'} \le L'_\block}}
  \{ \Block(z',t+T_\block) \text{ is well-started} \}.
\end{align}
In particular, the process $\widetilde\eta$ survives up to time
$t+T_\block$ in a good $\Block(z,t)$ and the region of the desired
density control by the profiles $\xi_n^-$ expands, see
Figure~\ref{fig:block}.

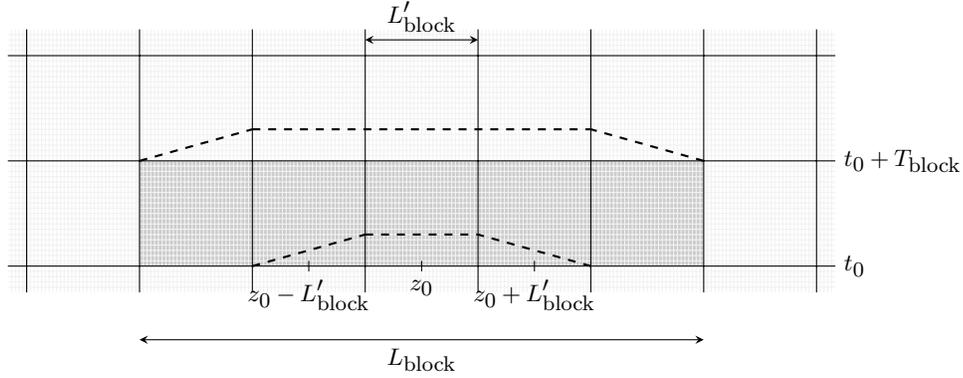
\begin{figure}
  \centering\qquad
  \begin{tikzpicture}[xscale=0.50,yscale=0.7]
    \fill [solid,gray!40] (-9,0) rectangle (6,2);

    \draw[step=0.1cm,black!5](-12.5,-.3) grid (9.5,4.4);
    \draw[xstep=3cm,ystep=2] (-12.5,-.3) grid (9.5,4.4);

    \node[below] at (-1.5,-0.2){$z_0$};
    \draw[] (-1.5,0.1)--(-1.5,-0.1);

    \node[below] at (1.5,-0.15){$z_0+ L'_{\mathrm{block}}$};
    \draw[] (1.5,0.1)--(1.5,-0.1);

    \node[below] at (-4.5,-0.15){$z_0- L'_{\mathrm{block}}$};
    \draw[] (-4.5,0.1)--(-4.5,-0.1);

    \node[right] at (9.5,0){$t_0$};
    \node[right] at (9.5,2){$t_0+T_{\mathrm{block}}$};

    \draw[stealth-stealth] (-3,4.25)--(0,4.2);
    \node[above] at (-1.5,4.35){$L'_{\mathrm{block}}$};

    \draw[stealth-stealth] (-9,-1.4)--(6,-1.4);
    \node[below] at (-1.5,-1.4){$L_{\mathrm{block}}$};


    \draw[thick,dashed]       (-3,0.6)--+(0:3)
    (-6,2.6)--+(0:9);

    \draw[thick,dashed]   (-6,0)--(-3,0.6)
    (3,0)--(0,0.6)
    (-9,2)--(-6,2.6)
    (6,2)--(3,2.6);
  \end{tikzpicture}
  \caption{Sketch of a $\Block(z_0,t_0)$ (dark grey), centred at the
    coarse-grained space-time lattice point $(z_0,t_0)$. The thick dashed
    lines depict the deterministic comparison density profiles $\xi_{t_0}^-(\cdot)$
    and $\xi_{t_0+T_{\mathrm{block}}}^-(\cdot)$ which have to be
    dominated by the density $\delta_R(\cdot;\widetilde\eta_n)$,
    $t_0\le n \le t_0+T_{\mathrm{block}}$ in order for the block to be
    \emph{good}. Note the picture is not drawn to scale: $L_{\block}$ and
    $L'_{\block}$ are both growing linearly in $R$ while $T_{\block}$
    does not grow with $R$.}
  \label{fig:block}
\end{figure}

By the construction
\eqref{eqn:etatilde} of $\widetilde \eta $, given $\mathcal{F}_t$,
if $\Block(z,t)$ is well-started, it can be decided
whether or not the event $\{ \Block(z,t) \text{ is good} \}$ occurs
for $(z, t) \in \mathbb{L}$ by inspecting (only) the values of
\begin{equation}
  \label{eqn:goodmeas}
  \big( U(x,n) \in \Z^d \times \N_0 :
  \norm{x - z} \le R_\block+ T_\block R, \, t < n \le t+ T_\block
  \big)
\end{equation}
(in fact, strictly speaking it suffices to observe the values of $U$'s at the
  space-time points
  $\{ (x,n) : \norm{x - z} \le R_\block+(t+T_\block-n)R, \, t < n \le t+ T_\block \}$).
Note that for $(z, t) \in \mathbb{L}$ and $(z', t) \in \mathbb{L}$ with
\begin{equation}
  \label{eqn:radius_z}
  \norm{z'-z} > L_\block + 2 T_\block R
  \approx (5+2/s) L'_\block \quad (\text{when } R \text{ is large})
\end{equation}
the space-time regions corresponding to \eqref{eqn:goodmeas} will be
disjoint.

Furthermore, by invoking Lemma~\ref{lem:dens.dev}(b)  we can uniformly
bound the probability of the density of $\widetilde \eta$ dominating the
comparison density profile $\xi^-$ for all space-time sites in
$\Block(z,t)$, which in turn yields
\begin{equation}
  \label{eqn:wellgood}
  \P\big(\Block(z,t) \text{ is good} \,\big|\, \mathcal{F}_t \big)
  \ge \ind_{\{\Block(z,t) \text{ is well-started}\}}\big(1- q(T_\block,R)\big)
\end{equation}
with
\begin{equation}
  \label{eqn:qn}
  q(T_\block,R) = 2 \abs[\big]{ \Block(z,t) }
  e^{- c V_R^d},
\end{equation}
which tends to $0$ as $R \to \infty$, since $\abs{\Block(z,t)}$ grows
only polynomially in $R$.

In order to make the comparison with oriented percolation, we define
random variables
\begin{equation}
  \label{eq:Yzt}
  Y(z,t) := \ind_{\{\Block(z,t) \text{ is good}\}}, \quad (z,t)\in \mathbb L,
\end{equation}
and say that $(z,t)\in \mathbb L$ is \emph{connected to infinity} in $Y$ if
there is a path $((z_i,t+i T_\block ):i\in \mathbb N_0)$ in $\mathbb{L}$
with $z_0=z$ and $\norm{z_i-z_{i-1}}\le L'_\block$ for all
$i\in \mathbb N$, such that $Y(z_i,t+i T_\block)=1$ for all
$i\in \mathbb N_0$ (such a path is called \emph{open} in $Y$). By the
argument above, it follows that if $(z,t)$ is well-started and connected
to infinity in $Y$, then the process $\widetilde \eta $ survives.

In order to show that the latter event occurs with positive probability,
we iteratively construct a coupling between the $Y(z,t)$'s from
\eqref{eq:Yzt} and a family $(\widetilde Y(z,t))_{(z,t)\in \mathbb L}$ of
i.i.d.~Bernoulli random variables with parameter $p(R)$ which satisfies
$p(R)\to 1$ as $R\to\infty$ such that we have
\begin{equation}
  \label{eq:Ydominates}
  Y(z,t) \ge \ind_{\{\Block(z,t) \text{ is well-started}\}} \widetilde Y(z,t)
  \quad \text{for all } (z,t) \in \mathbb{L} .
\end{equation}

We construct $\widetilde Y(\cdot,t)$ inductively over $t$ and begin with
a slightly informal description of this construction:
Assume that for some $t' \in T_\block \N$, a coupling satisfying
\eqref{eq:Ydominates} has been achieved for all
$(z,t) \in \mathbb{L}$ with $ T_\block \N \ni t < t'$. We then work
conditionally on $\mathcal{F}_{t'}$. The (random) set of nodes
\begin{equation*}
  W(t') := \{ z' \in L'_\block \Z^d : \Block(z',t') \text{ is well-started}\},
\end{equation*}
viewed as a graph where $z'$ and $z''$ are connected by an edge if the
space-time regions from \eqref{eqn:goodmeas} centred at $(z',t')$ and at
$(z'',t')$, respectively, overlap, is a locally finite graph with
uniformly bounded degrees. In fact, we see from \eqref{eqn:goodmeas} that
we have irrespective of the realisation of $\widetilde{\eta}_{t'}$ the
deterministic bound $(11+4/s)^d$ on the degree of any node (up to
  rounding, see~\eqref{eqn:radius_z}). Thus, by
\eqref{eqn:goodmeas}--\eqref{eqn:qn}, using well known stochastic
domination arguments for percolation models with finite-range
dependencies \cite{liggett1997domination}, it follows that the family
$(Y(z,t'))_{z \in W(t')}$ stochastically dominates a family
$(\widetilde Y(z,t'))_{z \in W(t')}$ of i.i.d.\ Bernoulli random
variables with parameter $p(R)$, where $p(R)\to 1$ as $R\to\infty$ and
the $(\widetilde Y(z,t'))_{z \in W(t')}$ are independent of
$\mathcal{F}_{t'}$ given $W(t')$, i.e.\ \eqref{eq:Ydominates} holds for
all $z \in W(t')$. In fact, $p(R)$ is a function of the maximal degree
$(11+4/s)^d$ of the dependence graph and the minimal guaranteed density
$1-q(T_\block, R)$ of good blocks, see Theorem 1.3 in
\cite{liggett1997domination}. For $z \not\in W(t')$,
\eqref{eq:Ydominates} imposes no condition at all on $\widetilde Y(z,t')$.
Thus, we can simply define $\widetilde Y(z,t') = \widehat Y(z,t')$ for
$z \not\in W(t')$ where $(\widehat Y(z,t))_{(z,t) \in \mathbb{L}}$ is an
independent family of i.i.d.\ Bernoulli($p(R)$) random variables.

In order to formalise this construction and, in particular, to show that
the random variables $\widetilde Y(z,t)$ are independent over different
time layers, note that by the construction \eqref{eqn:etacouplconstr}
from Lemma~\ref{lem:coupling_lemma}, we can write
\begin{equation*}
  Y(\cdot,t') = g\big(\eta_{t'}, (U(\cdot, n) : t' < n \le t'+T_\block)\big)
\end{equation*}
for some deterministic function
$g : \{0,1\}^{\Z^d} \times [0,1]^{\Z^d \times \{1,\dots,T_\block\}} \to \{0,1\}^{L'_\block\Z^d}$,
furthermore
$W(t') = W(\eta_{t'}) = \{ z \in L'_\block\Z^d : \Block(z,t') \text{ is well started}\}$.
For every $\zeta = (\zeta(z))_{z \in \Z^d} \in \{0,1\}^{\Z ^d}$,
\cite[Thm.~1.3]{liggett1997domination} and the discussion above provides
a coupling $\nu_\zeta$ of $\mathcal{L}(Y(\cdot,t') \,|\, \eta_{t'}=\zeta)$
and $\mathrm{Ber}(p(R))^{\otimes \Z^d}$ with the desired properties. We
can then disintegrate this joint law with respect to its first marginal
and describe the joint law $\nu_\zeta$ in a two-step procedure. It is
convenient to describe this via an auxiliary function
$h(\zeta;\cdot,\cdot)$ using additional independent randomness and obtain
that given $\eta_{t'} = \zeta$,
\begin{align*}
  Y(\cdot,t') = g\big(\zeta, (U(\cdot, n) : t' < n \le t'+T_\block)\big), \quad
  \widetilde Y(\cdot,t') = h\big(\zeta; Y(\cdot,t'), \widetilde{U}_{t'}\big)
\end{align*}
where $\widetilde{U}_{t'}$ is independent of everything else and
uniformly distributed on $[0,1]$ (see, for example, Theorem~5.10 in
  \cite{Kallenberg1997foundations}).
By construction, since $U(\cdot,n), n > t'$ and
$\widetilde{U}_{t'}$ are independent of $\mathcal{F}_{t'}$, we have for
$A \in \mathcal{F}_{t'}$ and measurable $B \subseteq\{0,1\}^{L'_\block\Z^d}$
\begin{align*}
  \P\big(A \cap \{ \widetilde{Y}(\cdot,t') \in B\} \big)
  &= \E\Big[ \ind_A \P\big( h(\eta_{t'}; Y(\cdot,t'), \widetilde{U}_{t'}) \in B
      \,|\, \mathcal{F}_{t'}\big) \Big]\\ &= \P(A) \mathrm{Ber}(p(R))^{\otimes L'_\block\Z^ d}(B) .
\end{align*}
This shows the required independence of $\widetilde Y$ and completes the
induction step.

We see from \eqref{eq:Yzt}, \eqref{eq:Ydominates} and
\eqref{goodmakeswell} that every open path in $\widetilde Y(\cdot,\cdot)$
is automatically also an open path in $Y(\cdot,\cdot)$. Furthermore, by
well known properties of oriented site percolation, we have
\begin{equation*}
  \P\big((z,t) \text{ is connected to infinity in } \widetilde Y\big)
  = \P\big((0,0) \text{ is connected to infinity in } \widetilde Y\big)
  > 0
\end{equation*}
if $p(R)$ is sufficiently close to $1$, i.e.~for all $R$ large enough.

To conclude, let $\eta_0$ be any initial configuration containing at
least one particle, and let $\widetilde\eta_0 = \eta_0$. It is then easy
to see (as this involves requiring only finitely many random variables
  $U(x,n)$ to be sufficiently small), that one can find
$(z,t) \in \mathbb{L}$, so that the probability that $\Block(z,t)$
is well-started is positive.

Therefore, due to the above properties,
\begin{align*}
  \mathbb P(\eta \text{ survives})
  &\ge \mathbb P( \widetilde\eta  \text{ survives})
  \\&\ge \E\big[ \ind_{\{\Block(z,t) \text{ is well-started}\}}
    \ind_{\{(z,t)\text{ is connected to infinity in $Y$}\}} \big]
  \\&\ge \mathbb P\big(\Block(z,t) \text{ is well-started}\big)
  \mathbb P\big((z,t)\text{ is connected to infinity in $\widetilde Y$}\big)
  \\&>0 \qquad \text{for all $R$ large enough},
\end{align*}
which completes the proof of Theorem~\ref{thm:survival_eta}.
\smallskip

The proof of Corollary~\ref{cor:survival_eta_cor} follows directly from the properties of the percolation cluster and our definition of good blocks: using the coupling with supercritical oriented percolation (on the level of space-time blocks) constructed above, \eqref{eq:positivedensity} follows from the fact that given that it survives, the cluster of the origin in supercritical oriented percolation has positive density. This can be seen for example via a coupling with a supercritical discrete time contact process started from its upper invariant measure (see e.g.\ \cite[Prop.~6]{DurrettGriffeath1982} and the discussion around Lemma~2.9 in \cite{BirknerCernyDepperschmidt2016}).

\section{Complete convergence}
\label{sec:complete_convergence}

In this section we show our main results in the regime where the particle
system survives with a positive probability and is well approximated by
the deterministic coupled map lattice introduced in
Section~\ref{sect:CML}. In particular, we assume that $\mu  \in (1,e^2)$
and $R$ is large enough. In Section~\ref{sect:proof:thm:coupling}, we
start with Theorem~\ref{thm:coupling} providing the coupling of processes
started with different initial conditions.
Theorem~\ref{thm:complete_convergence} is then shown in
Section~\ref{sec:prfcompconv}.

\subsection{Coupling construction: Proof of Theorem~\ref{thm:coupling}}
\label{sect:proof:thm:coupling}

As in
Section~\ref{sec:survival}, the central ingredient
will be a
block construction and then a suitable comparison with oriented
percolation. The definition of ``good blocks'' will be  more involved than in
Section~\ref{sec:survival} and is inspired by the construction in
\cite[Section~5]{birkner2007survival}.

In brief, the construction of a good block around $z$ is as follows. We
consider a (large) ball $B$ around $z$ and assume that $\eta^{(1)}$ and
$\eta^{(2)}$ agree on $B$ and the respective $R$-densities of the two
processes are close to $\theta_{\mu}$. On an even larger ball $B'$ we add
milder and milder requirements (as the distance from the centre
  increases) on the densities of the processes. The contraction property
of $\varphi_\mu$, see Lemma~\ref{lem:alpha_sequences} below, together
with the concentration property of the densities of $\eta^{(i)}$
guaranteed by Lemma~\ref{lem:dens.dev} then ensure that the area in which
the $\eta^{(1)}$ and $\eta^{(2)}$ are coupled expands in time with high probability. In
order to guarantee survival of the processes we also require that the
respective densities of $\eta^{(1)},\eta^{(2)}$ dominate the
deterministic comparison density profile as defined in \eqref{eqn:xi_n}
(the latter was also used in Section~\ref{sec:survival}).

We now proceed with the formal definitions. Throughout this section, we
again use the coupling construction from Section~\ref{ss:coupling}:
Given two initial conditions $\eta^{(1)}_0$ and $\eta^{(2)}_0$, we
construct both $(\eta^{(1)}_n)_n$ and $(\eta^{(2)}_n)_n$ using
\eqref{eqn:etacouplconstr} with the same $U(x,n)$'s, that is, we set
\begin{equation}
  \label{eqn:coupling}
  \eta_{n+1}^{(i)}(x) = \ind_{\{ U(x,n+1)
      \le \varphi_{\mu}(\delta_R(x;\eta_n^{(i)}))\}},
  \quad i \in \{ 1,2 \}, \quad (x,n)\in \Z^d\times \N_0.
\end{equation}
Since we are from now on interested in two copies of the branching annihilating process, we redefine the filtration $(\mathcal{F}_n)$ from \eqref{eqn:filtration_Fn} by including both initial conditions, i.e.\
\begin{equation*}
    \mathcal F_n := \sigma\big( \eta^{(i)}_0(x) : x\in \mathbb Z^d, i=1,2 \big) \vee \sigma \big(U(x,j): x\in \mathbb Z^d, j\le n\big).
\end{equation*}
It is clear that this updated filtration is finer than the natural filtration of the two processes, in the sense that for all $n\ge0$, it holds that $ \mathcal{F}_n \supseteq\sigma \big(\eta_j^{(i)}(x): x\in \mathbb Z^d, j\le n, i=1,2\big).$

In order to define the comparison density profiles that are used to determine
whether a block is good, we need a simple lemma which gives some useful
properties of the function $\varphi_\mu $ in the vicinity of its
non-trivial fixpoint $\theta_\mu$. The result is fairly standard, we
provide a proof for completeness' sake in Section~\ref{sec:construction}
(cf.~also \cite[Proof of~Lemma~12]{birkner2007survival}).

\begin{lemma}
  \label{lem:alpha_sequences}
  For every $\mu\in (1,e^2)$ there is $\varepsilon>0$ and
  $\kappa(\mu,\varepsilon)<1$ such that $\varphi_\mu $ is a contraction on
  $[\theta_{\mu}-\varepsilon, \theta_{\mu}+\varepsilon]$, that is,
  \begin{equation*}
    |\varphi_{\mu}(w_1)-\varphi_{\mu}(w_2)|
    \le \kappa(\mu,\varepsilon)|w_1-w_2|
    \qquad
    \text{for }w_1, w_2 \in
    [\theta_{\mu}-\varepsilon, \theta_{\mu}+\varepsilon].
  \end{equation*}
  Moreover, there exist a strictly increasing sequence
  $\alpha_m \uparrow \theta_\mu$ and a strictly decreasing sequence
  $\beta_m \downarrow \theta_\mu$ such that
  $\varphi_{\mu} ([\alpha_m, \beta_m]) \subseteq(\alpha_{m+1}, \beta_{m+1})$
  for all $m \in \N$. Furthermore, it is possible to choose $\alpha_1 >0$
  arbitrarily small and $\beta_1 > 1/e$.
\end{lemma}
We now take $b$ as in~\eqref{eqn:psi} and fix $\varepsilon $, $\kappa (\mu ,\varepsilon )<1$, as well as sequences $\alpha_m \uparrow \theta_\mu$,
$\beta_m \downarrow \theta_\mu$ as in Lemma~\ref{lem:alpha_sequences}, with $\alpha_1=b $ and $\beta_1>1/e$.
Then we choose $m_0$ such that $\beta_m - \alpha_m < \varepsilon$ for every
$m \ge m_0$. These choices will remain fixed
throughout the remainder of this section.

Next define the size of the blocks
\begin{equation}
  \label{def:Lblock}
  L'_\block := 2 \lceil R \log R \rceil, \qquad
  L_\block := c_{\sspace} L'_\block \quad
  \text{and} \quad T_\block := c_{\ttime} \lceil \log R \rceil,
\end{equation}
where $c_\ttime > -(d+1)/\log \kappa(\mu, \varepsilon)$ and $c_\sspace = 4(1 + c_\ttime)$ are integer constants. Remark~\ref{rem:constants} below explains these choices.
As in Section~\ref{sec:survival}, we
introduce $R'_\block = L'_\block/2$ and $R_\block = L_\block/2$ for the
radii of the blocks, and,  for $(z,t)$ in the sub-lattice
$\mathbb L:= L'_\block \mathbb Z^d \times T_\block \mathbb N_0$, we define
\begin{equation*}
  \Block(z,t) := \big\{ (x,n) \in \Z^d \times \N_0 :\norm{x - z} \le R_\block,
    \, t \le n \le t+ T_\block \big\}.
\end{equation*}
Further, let us specify the radius for which the strongest form of density control, alluded to in the above informal description, holds. More precisely set $c_\dens = 1+2 c_\ttime$ and $R_{\mathrm{dens}} := 2 c_{\mathrm{dens}} R'_\block$. Again, the discussion on the choice of $c_\dens$ is postponed to Remark~\ref{rem:constants}.

Recall the functions $\xi^-_n(x)$ defined in \eqref{eqn:xi_n}. We use
them here with 
$R_{\mathrm{init}} = R_{\mathrm{dens}} + m_0 R$ in \eqref{eqn:xitilde}.
For $k \in \{ 0, \dots, T_\block \}$ set
$R_\dens(k)= R_\dens+ k \lceil s R \rceil $, then let
\begin{align*}
  \zeta^-_k(x)
  := \begin{cases}
    \alpha_{m_0}  & \text{if } \norm{ x } \le R_{\mathrm{dens}}(k) \\
    \alpha_{m_0 -j+1}
    & \text{if } R_{\mathrm{dens}}(k) + (j-1) R < \norm{ x }
    \le R_{\mathrm{dens}}(k) + j R, \,1 \le j \le m_0 \\
    \xi^-_k(x) & \text{if } \norm{ x } > R_{\mathrm{dens}}(k) + m_0 R,
  \end{cases}
\end{align*}
and
\begin{align*}
  \zeta^+_k(x)
  := \begin{cases}
    \beta_{m_0}  & \text{if } \norm{ x } \le R_{\mathrm{dens}}(k) \\
    \beta_{m_0 -j+1} & \text{if } R_{\mathrm{dens}}(k) + (j-1) R
    < \norm{ x } \le R_{\mathrm{dens}}(k) + j R,
    \,1 \le j \le m_0 \\
    1 \vee \beta_1 & \text{if } \norm{ x } > R_{\mathrm{dens}}(k) + m_0 R.
  \end{cases}
\end{align*}
See also Figure~\ref{fig:blocks_complete_conv}.

\begin{figure}[b!]
  \centering
  \begin{tikzpicture}[scale = 0.83]

    \draw[-stealth] (-0.3,0)--+(0:14.2);
    \draw[-stealth] (-1.5,0)--+(90:7);
    \draw (-1.5,0) --+(0:0.3);
    \foreach \j in {1,...,9}
    \draw 	(\j*3/2,0)--+(-90:.1);
    \draw[stealth-stealth]	(9,-.2)--+(-180:3/2);
    \draw[stealth-stealth]	(9+1/100,-.9)--+(0:3*3/2-1/100);
    \draw[stealth-stealth]	(0,-.9) --+ (0:9-1/100);
    \draw (-1.5,3.5)--+(180:.1) node[left]{$\theta_{\mu}$};
    \draw[dashed, gray!40]
    (-1.5,6)--+ (0:.3)		    
    (-.3,6)--+ (0:3+5*3/2-1.2)		    
    (9,0)--+ (90:6);
    \draw[dashed,thick, black!20!red]	(-1.5,3.5) --+ (0:.3)
    (-.3,3.5) --+ (0:13.8);		
    \draw
    (-1.5,1)--+(180:.1) node[left]{$\alpha_1$}
    (-1.5,2) + (180:.1) node[left]{$\vdots$}
    (-1.5,3)--+(180:.1) node[left]{$\alpha_{m_0}$}
    (-1.5,6)--+(180:.1) node[left]{$\beta_{1}$}
    (-1.5,5) +(180:.1) node[left]{$\vdots$}
    (-1.5,4)--+(180:.1) node[left]{$\beta_{m_0}$}
    (0,0)--+(-90:.1) node[below]{$R_{\mathrm{dens}}(k)$}
    (-1.5,0)--+(-90:.1) node[below]{$0$};

    \node[below] at (4.5,-.9){$m_0 R$};
    \node[below] at (9-3/4,-.2){$R$};
    \node[below] at (9-3/4+3,-.9){$\lceil wR \rceil$};

    \draw[gray!40, dashed] (0,1 )--+ (0:9);
    \draw[gray!40, dashed] (-1.5,1 )--+ (0:0.3);
    \draw[gray!40, dashed] (-.3,1 )--+ (0:0.3);

    \foreach \j in {3,4}
    {
      \draw[orange,thick] (-1.5,\j )--+ (0:0.3);
      \draw[orange,thick] (-.3,\j )--+ (0:0.3);
    }
    \foreach \j in {-1,0,1}
    {
      \filldraw[gray!40] (-0.8+\j*0.2,1) circle (0.4pt);
      \filldraw[orange] (-0.8+\j*0.2,4) circle (0.4pt);
      \filldraw[orange] (-0.8+\j*0.2,3) circle (0.4pt);
      \filldraw[black] (-0.8+\j*0.2,0) circle (0.4pt);
      \filldraw[gray!40] (-0.8+\j*0.2,6) circle (0.4pt);
      \filldraw[black!20!red] (-0.8+\j*0.2,3.5) circle (0.4pt);
    }

    \draw[black!30!green, thick]	(9,1)--(13.5,0.1);
    \draw[black!30!green, thick]	(13.5,0.1)--+(-90:0.1);
    \draw[orange, thick] 	(0,3) --+ (0:1)
    (0,4)--+ (0:1)
    (9,6)--+(0:4.5);

    \foreach \j in {0,...,4}
    \draw[orange, thick] 	(\j*3/2,3-\j/10-\j*\j/20-\j/20)--+(0:3/2)
    (3/2+\j*3/2,3-\j/10-\j*\j/20-\j/20)--+(-90:\j/10+2/10);
    \draw[orange, thick, ] (9-3/2,1)--+(0:3/2);

    \foreach \j in {0,...,4}
    \draw[orange, thick] 	(\j*3/2,4+\j/10+\j*\j/20+\j/20)--+(0:3/2)
    (3/2+\j*3/2,4+\j/10+\j*\j/20+\j/20)--+(90:\j/10+2/10);
    \draw[orange, thick] (6+3/2,6)--+(0:5/2);
  \end{tikzpicture}
  \caption{The part of the deterministic comparison density
    profiles $\zeta_{k}^+$ and $\zeta_{k}^-$ (in orange and green) left from
    $R_{\mathrm{dens}}(k)$.
    In a \emph{good} block the densities of both $\eta^{(1)}$ and
    $\eta^{(2)}$ stay between the union of the orange lines and the green
    profile, which is glued to the bottom orange profile. The green line
    is the (suitably recentred and shifted) profile of $\xi^-$, which we
    introduced to prove survival.}
  \label{fig:blocks_complete_conv}
\end{figure}
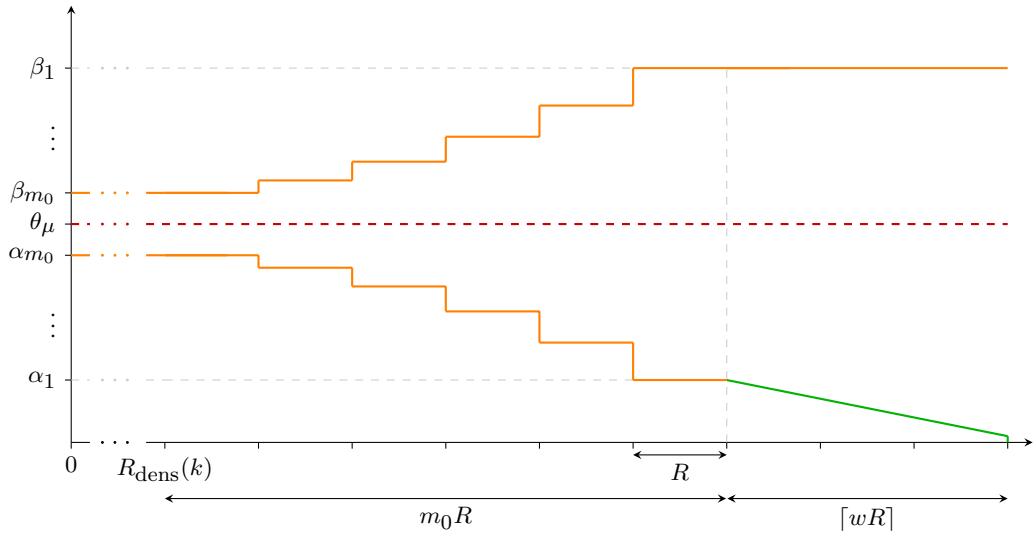

The functions $\zeta^-_k(\cdot) < \zeta^+_k(\cdot)$ are comparison density
profiles in the sense of Definition~\ref{def:density_profiles}, in particular, they
satisfy the following analogue of \eqref{eqn:xi.comp.bd1}.

\begin{lemma}
  \label{lem:dens.dev2}
  There exists $\delta > 0$ with the following property:
  For $k \in \N_0$ and any $(\zeta(x))_{x \in \Z^d} \in [0,1]^{\Z^d}$
  satisfying $\zeta^-_k \le \zeta \le \zeta^+_k$ on  $\Supp(\zeta^-_k)$,
  it follows that
  \begin{equation*}
    (1+\delta) \zeta^-_{k+1}(x) \le V_R^{-d}
    \sum_{y \in B_R(x)} \varphi_\mu\big( \zeta(y) \big)
    \le (1-\delta) \zeta^+_{k+1}(x) \quad
    \text{for all } x \in \Supp(\zeta^-_{k+1}).
  \end{equation*}
\end{lemma}

\begin{proof}
  For $x$ such that $\zeta_k^{-}(x)$ agrees with the previously defined
  profile $\xi_k^-(x)$ the lower bound in the statement follows easily from
  Lemma~\ref{lem:xi_n}.

  Let $x \in \Supp(\zeta^-_{k+1})$ with $\zeta_{k+1}^-(x) = \alpha_{j}$,
  for some $2\le j \le m_0$. Then
  \begin{align*}
    \zeta_k^-(y) =  \alpha_{j} \text{ if } y \in B_R(x) \cap \Upsilon_{R}
    \qquad \text{and} \qquad
    \zeta_k^-(y) \ge \alpha_{j-1}
    \text{ if } y \in B_R(x) \cap \Upsilon_{R}^c,
  \end{align*}
  where
  $\Upsilon_{R} := \{ z:R_\dens(k  )+(j-1)R \le \norm{ z } \le R_\dens(k)+jR \}$. Note that $\abs{\Upsilon_R\cap B_R(x)}\ge c V_R^d$ for some $c>0$, uniformly in the $x$ we consider here.
  The properties of sequences $\alpha_m, \beta_m$ from
  Lemma~\ref{lem:alpha_sequences} then imply that there exists $\delta>0$
  (depending on  $(\alpha_m)_{m \le m_0}$,
    $(\beta_m)_{m \le m_0}$ and $d$) such that
  \begin{align*}
    V_R^{-d} \sum_{y \in B_R(x)} \varphi_{\mu}(\zeta(y))
    &\ge \alpha_{j+1} \abs{B_R(x) \cap \Upsilon_{R}}  V_R^{-d}
    + \alpha_{j} \abs{ B_R(x) \cap \Upsilon_{R}^c}
    V_R^{-d} \ge \alpha_{j}(1+\delta)
  \end{align*}
  and similarly for the upper bound.  This completes the claim for the
  remaining parts of the profile (those in orange in
    Figure~\ref{fig:blocks_complete_conv}).
\end{proof}
We proceed in a similar fashion as in
Section~\ref{sec:survival} and introduce a new notion of well-started and of good blocks. 
These updated definitions involve two
copies $\eta^{(1)}$, $\eta^{(2)}$ of the system. A well-started block is now determined by the local density of
the true system being controlled by the $\zeta^+_k, \zeta^-_k$ profiles, in addition to which we require agreement of the true processes in the central part of the block. 
\begin{definition}
  \label{def:well_started}
  A $\Block(z,t)$ based at $(z,t) \in \mathbb{L}$ is \emph{well-started}
  if
  \begin{align}
    \label{eq:well-started2}
    \delta_R(x; \eta_t^{(i)}) \in \big[ \zeta^-_0(x-z), \zeta^+_0(x-z) \big]
    \quad \text{for all } x \in  z+\Supp(\zeta^-_0), \; i=1,2
  \end{align}
  and
  \begin{align}
    \label{eqn:well-started2a}
    \eta_t^{(1)}(x) = \eta_t^{(2)}(x) \quad \text{for all } x \in B_{R'_\block}(z).
  \end{align}
\end{definition}
Again as in Section~\ref{sec:survival} we use this as the starting point off of which we base our notion of goodness as the spreading of the control given by well-startedness to
neighbouring regions.
\begin{definition}
  \label{def:good_block}
  We call a $\Block(z,t)$ based at $(z,t) \in \mathbb{L}$ \emph{good} if
  \begin{enumerate}[(i)]
    \item $\Block(z,t)$ is \emph{well-started},

    \item $\eta_{t+T_\block}^{(1)}(x) = \eta_{t+T_\block}^{(2)}(x)$  for
    $\norm{x-z} \le 3 R'_\block$,

    \item $\big(\eta^{(1)}_{t+T_\block}, \eta^{(2)}_{t+T_\block}\big)$
    satisfy \eqref{eq:well-started2} centred at $z + L'_\block e$ for all $e \in B_1(0)$, that is
    \begin{equation*}
    \delta_R(x; \eta_{t+t_\block}^{(i)}) \in \big[ \zeta^-_0(x-z-L'_\block e), \zeta^+_0(x-z-L'_\block e) \big]
    \end{equation*}
    for all $x \in  z+L'_\block e+\Supp(\zeta^-_0)$, for all $e \in B_1(0)$, for $i=1,2$.
  \end{enumerate}
\end{definition}
Property (iii) implies that if $\Block(z,t)$ is good, then
$\Block(z+ L'_\block e,t+T_\block)$ will be well-started for all
$\norm{e} \le 1$.

\begin{remark}
\label{rem:constants}
Let us now comment on our choice of the constants $c_\sspace$, $c_\dens$, $c_\ttime$. It is instructive to first give $c_\sspace$ as a function of $c_\dens$, then $c_\dens$ as a function of $c_\ttime$, and ultimately fixing $c_\ttime$ large enough.

\begin{enumerate}[(i)]
    \item Note first that $\zeta_0^\pm$ are constant on a box of size $R_\dens$ (which is of order $R \log R$) and then increase (resp. decrease) on boxes with length of order $R$. It follows readily that $\Supp(\zeta_0^-)\subseteq B_{2R_{\mathrm{dens}}}(0)$ for large enough $R$. Therefore $2c_\dens$ blocks of size $L'_\block$ fully cover the spatial region determining whether a block is well-started. Furthermore, we need to to provide additional space for the well-started configurations to spread to in time $T_\block$. This warrants the choice $c_{\sspace} = 2 c_{\dens} +2$. Note in this context that a
much smaller $L_\block$ would suffice, but defining it to be a multiple
of $L'_\block$ gives a more convenient notation.

    \item In order to have a \emph{well-started} block at $(z,t)$ for which property~\eqref{eqn:well-started2a} spreads to a region of
radius $3R'_\block$ around $z$ in time $T_\block$, the region of space
around $z$ for which the densities of $\eta^{(1)}, \eta^{(2)}$ are near
$\theta_\mu$ must be large enough.
As will be seen later on (see Section~\ref{sec:proofdetails.coupl}) this is due to the crucial role that the contraction property of Lemma~\ref{lem:alpha_sequences} plays in the expansion of the coupling and translates loosely to 
$R_{\mathrm{dens}}$ being large enough, namely
\begin{equation*}
  R_{\mathrm{dens}} > R'_\block + T_{\block} \lceil s R \rceil + T_{\block}R.
\end{equation*}
This can also be seen as an incentive for taking $T_\block$ to be of order $\log R$ and $R_\dens$ to be of order $R \log R$. Further it shows that  $c_{\mathrm{dens}}$ needs to be chosen suitably large; it suffices to take $c_{\mathrm{dens}} = 1+2 c_{\mathrm{time}}$.

    \item 
Assume that on the event that a block at $(z,t)$ is well started, property (ii) of Definition~\ref{def:good_block} does not hold, i.e.\ there is a site at the top of the block at which $\eta^{(1)}$ and $\eta^{(2)}$ disagree. As will be seen in Section~\ref{sec:proofdetails.coupl}, the probability of the two processes disagreeing at a site (in a well-started block) decays by a factor of $\kappa(\mu,\varepsilon)$ at each time step, when tracing the unsuccessful coupling backwards in time though the block. By a union bound, it follows that the probability that a well-started block at $(z,t)$ does not satisfy (ii), is bounded by $\kappa(\mu,\varepsilon)^{T_\block} $
multiplied by the number of sites that are within distance $R'_\block + T_\block \lceil sR \rceil$ of $z$. 
For this probability to decay in $R$, the constant $c_\ttime$ must satisfy $ c_{\mathrm{time}} > -(d+1)/\log \kappa(\mu, \varepsilon)$.
\end{enumerate}
\end{remark}
In order to set up comparison with oriented percolation, in the same
fashion as in Section~\ref{sec:survival}, we need to show that the \emph{good} blocks have high density and that the block
dependencies have finite range that does not depend on $R$. To this end,
note first that the event $\{ \Block(z,t) \text{ is good} \}$ depends (only) on
$\{ \eta^{(i)}_t(x), x \in B_{R_\block}(z), i=1,2 \}$ and
$\{ U(y,t+k) : y \in B_{3 R_\block}(z), k=1,2,\dots,T_\block\}$.

\begin{lemma}
  \label{lem:welltogood}
  For $(z,t) \in \mathbb{L}$,
  \begin{align*}
    \P\big( \text{\rm $\Block(z,t)$ is good} \,\big|\, \mathcal{F}_t \big)
    \ge
    \ind_{\{ \text{\rm $\Block(z,t)$ is well-started}\}}
    \big(1-q(R,\mu)\big)
  \end{align*}
  with $q(R,\mu)\to 0$ as $R \to \infty$.
\end{lemma}

\noindent
See Section~\ref{sec:proofdetails.coupl} for the proof.

Equipped with Lemma~\ref{lem:welltogood} we can repeat the comparison
construction from Section~\ref{sec:survival} and obtain the analogues of
\eqref{eq:Yzt} and \eqref{eq:Ydominates} in our context. That is we define $Y(z,t) = \ind_{\{\Block(z,t) \text{ is good}\}}$, and then 
couple $(\eta^{(1)}, \eta^{(2)})$
with a (high density) i.i.d.\
Bernoulli field $(\widetilde Y(z,t))_{(z,t) \in \mathbb{L}}$ such that
\begin{align*}
  Y(z,t)
  \ge \ind_{\{\Block(z,t) \text{ is well-started}\}} \widetilde Y(z,t)
  \quad \text{for all } (z,t) \in \mathbb{L}
\end{align*}
and $p(R) = \P(\widetilde Y(z,t)=1) \to 1$ as $R\to\infty$.

This shows that the density of good blocks (and thus also the density of
  space-time sites where $\eta^{(1)}$ and $\eta^{(2)}$ agree) will be
high. In order to conclude that in fact $\eta^{(1)}$ and $\eta^{(2)}$
will agree a.s.\ from some time on in a growing space-time region, we
invoke the fact that ``dry'' ($\widehat{=}$ ``uncoupled'') clusters of blocks
do not percolate when $p(R)$ is close to $1$. More precisely we set
\begin{equation*}
   C_0 := \left\{ (z,t)\in \mathbb{L}  : \;
      \parbox{0.6\linewidth}{There exists a path $(z_0,0),(z_1,T_\block),\dots, (z_t,t)$ in $\mathbb{L}$ with $z_0 = 0, z_t = z$ such that $\norm{z_i-z_{i-1}} \le L'_\block$  and $\widetilde Y(z_i,i T_\block) = 1$ for $i \in \{1,\dots,t/T_\block\}$} \: 
      \right\}
\end{equation*}
to be the cluster of sites which are connected to the origin by an open path in the Bernoulli field $(\widetilde Y(z,t))_{(z,t) \in \mathbb{L}}$.  Further we say that a space-time point $(z,t) \in \mathbb{L}$ is $C_0$-\emph{exposed} if there is an arbitrary path from it to the zero-time slice, which entirely avoids $C_0$, i.e.\ if there is a path 
$(z_0,0),\dots,(z_t,t)$ in $\mathbb{L}$ with $z_t = z$ such that $\norm{z_k-z_{k-1}}\le L'_\block$ and $(z_k,k T_\block) \notin C_0, k =1,\dots, t/T_\block$.

It follows from \cite[Section~3]{durrett1992multicolor} that there is a
truncated cone originating from the origin in which there exist no $C_0$-exposed
sites. The exact statement we are interested in is a direct reformulation
of \cite[Lemma~14]{birkner2007survival}.

\begin{lemma}[{\cite[Lemma~14]{birkner2007survival}}]
  \label{lem:birkner_survival}
  If $p(R)$ is sufficiently close to $1$, then there is a positive
  constant $c>0$ and an almost surely finite random time $\tau$, such
  that conditioned on $\{{|C_0| = \infty}\}$ there are no $C_0$-exposed
  sites in $\{(z,t) \in \mathbb{L} : \norm{z} \le ct, t\ge \tau \}$.
\end{lemma}

For large enough $R$ the Bernoulli field
$(\widetilde Y(z,t))_{(z,t) \in \mathbb{L}}$ contains an infinite
cluster of open sites with probability one. Similarly to Section~\ref{sec:survival}, because a \emph{good} block will be created with positive probability from any non-trivial initial condition, we can assume without loss of
generality that this cluster contains the origin and that the block at the origin is good.

Lemma~\ref{lem:birkner_survival} together with Lemma~\ref{lem:welltogood} imply that for sufficiently large $R$,  on
$\{ \abs{ C_0}  = \infty \}$ there is a (random) time $\tau>0$ and a constant $c>0$ such that no sites in $ \{ (z, t) \in \mathbb{L} : \norm{z}\le ct,t\ge \tau\}$ are  $C_0$-exposed.
We show that this implies that
$\eta^{(1)}$ agrees with $\eta^{(2)}$ on the space-time cone $A := \{ (z,t) \in \Z^d\times \N : \norm{z}\le c(t-\tau), \, t\ge \tau \}$ centered at $(0,\tau)$. 
Indeed, assume to the contrary that there exists  $(z,s) \in A$ such that
$\eta_s^{(1)}(z) \neq \eta_s^{(2)}(z)$. Then we can find a path
$(z,s), (x_{s-1},{s-1}),\dots,(x_0,0)$ in $\Z^d\times \N_0$ such
that $x_u \in B_{R}(x_{u+1})$ and 
$\eta_u^{(1)}(x_u) \neq \eta_u^{(2)}(x_u)$ for all $0\le u \le s-1$. By
disregarding all $u$'s which are not a multiple of $T = T_\block$, there
exists some integer $k$ and a sub-path
$(z,s),(x_{kT},kT),\dots, (x_0,0)$ in $\Z^d \times \N_0$ ``backwards
in time''. Assume without loss of generality that $s$ is a multiple of $T$
and associate to the sub-path the nearest neighbour path
$\big((Z, k+1),(X_k,k),\dots, (X_0,0)\big) \subseteq \mathbb{L}$
where $Z,X_k \in L'_{\mathrm{Block}}\Z^d$ are the respective closest grid-points to
$z$ and $x_{kT}$ in the coarse-grained lattice. In particular
$\norm{X_k-x_{kT}}\le R'_{\mathrm{Block}}$ for $k = 0,\dots, s$. By
definition $\widetilde Y(X_k,kT) =0$ for $k = 0,\dots, s$, whence
$(Z,s)$ is a $C_0$-exposed site, contradicting
Lemma~\ref{lem:birkner_survival} and yielding that in fact $\eta_s^{(1)}(z) = \eta_s^{(2)}(z)$. As $(z,s) \in A$ was chosen arbitrarily, the claim of Theorem~\ref{thm:coupling} follows with $T^{\mathrm{coupl}} = \tau$ and $a = a(R,\mu,d) = c$.

\subsection{Proof of Lemma \ref{lem:welltogood}}
\label{sec:proofdetails.coupl}

We now show the key step in proving Theorem~\ref{thm:coupling}, which is showing that
the coupled region in the \emph{well-started} configuration of a \emph{good}
block expands to the neighbouring sites with high probability. 

\begin{proof}[Proof of Lemma \ref{lem:welltogood}] In order to keep the notation lighter we only prove the lemma for a block
centred at the origin at time $0$. That is, we show that for some $q=q(R,\mu)\to 0$ as $R \to \infty$,
\begin{equation}
\label{eq:welltogood_zero}
\P\big( \text{\rm $\Block(0,0)$ is good} \,\big|\, \mathcal{F}_0 \big)
    \ge
    \ind_{\{ \text{\rm $\Block(0,0)$ is well-started}\}}
    \big(1-q(R,\mu)\big)
\end{equation}
Shifting the block yields the desired
property for blocks centred at arbitrary space-time sites. Note that we still condition on $\mathcal{F}_0$, as we allow for possibly random initial configurations $\eta_0^{(i)}, i = 1,2$. As was already 
anticipated in Remark~\ref{rem:constants}, in order to see the spreading of the coupling after $T_{\block}$ steps, we need a
large number of sites within distance $R_\dens$ of the origin for which the
densities of both $\eta^{(1)}, \eta^{(2)}$ are close to $\theta_{\mu}$. 
This is made precise by the following auxiliary events, where the
densities have the prescribed behaviour on balls whose radii decrease by
$R$ at each time step.

Recall that $T_\block = c_\ttime \lceil\log R\rceil$ and write
$R'(k) := R'_\block + k \lceil sR \rceil$. For $n \in \N$ let
\begin{align*}
  \Psi_n := \Big\{ |\delta_R(x;\eta_j^{(i)})-\theta_{\mu}|<\varepsilon,
    \, \forall x : \norm{ x } \le  R'(n) + (n-j)R, \,
    \forall j \in \{ 1, \dots, n \}, \, i \in \{ 1,2 \} \Big\}.
\end{align*}
(Recall also that $\varepsilon$ was chosen at the beginning of
  Section~\ref{sect:proof:thm:coupling}, above \eqref{def:Lblock}.)

\begin{figure}
  \centering
  \begin{tikzpicture}[scale = 0.6]
    \draw[xstep=0.3cm,ystep=0.2,black!5] (-10.7,-.5) grid (10.7,4.5);

    \draw[black] (0,-0.1)--+(-90:0.2);
    \node[below] at (0,-0.3){$0$};
    \draw[black] (0,4.1)--+(-90:0.2);

    \draw[stealth-stealth] (0,4.25)--(2.1,4.25);
    \node[above] at (1,4.35) {$R'(n)$};

    \draw[stealth-stealth] (0,-1.2)--(2.1+20*0.3,-1.2);
    \node[below] at (1+10*0.2,-1.2) {$R'(n)+(n-1)R$};

    \draw[black] (-2.1,4)--(2.1,4);
    \draw[black]   (-2.1-20*0.3,4-21*0.2)--(2.1+20*0.3,4-21*0.2);

    \foreach \j in {0,...,19}
    {
      \draw[black] (2.1+\j*0.3,4-\j*0.2)--+(-90:0.2)
      (2.1+\j*0.3,4-\j*0.2-0.2)--+(0:0.3);
      \draw[black] (-2.1-\j*0.3,4-\j*0.2)--+(-90:0.2)
      (-2.1-\j*0.3,4-\j*0.2-0.2)--+(-180:0.3);
    }
    \draw[black] (2.1+20*0.3,4-20*0.2)--+(-90:0.2);
    \draw[black] (-2.1-20*0.3,4-20*0.2)--+(-90:0.2);
  \end{tikzpicture}

  \caption{The event $\Psi_n$ occurs if the local density of
    $\eta^{(i)},i = 1,2$ is within $\varepsilon$ distance of the fixed
    point $\theta_{\mu}$ for all space-time points in the above pyramid.
    For convenience of presentation the spatial axis in the sketch is
    scaled by $R$, while the temporal axis is not scaled.}
  \label{fig:pyramid}
\end{figure}
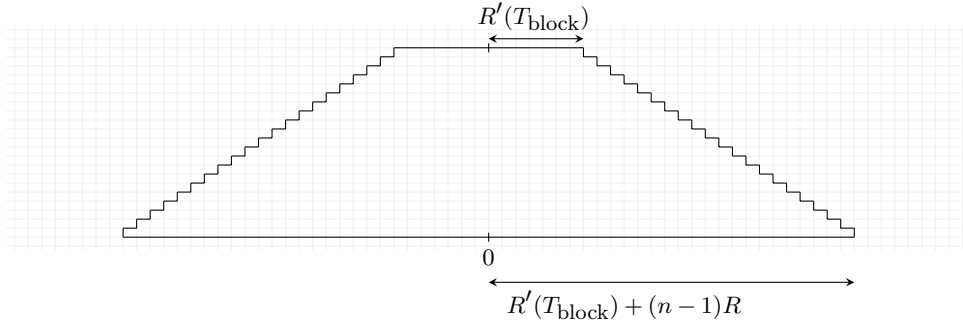

Note that in a \emph{well-started} configuration around the origin we
have $|\delta_R(x;\eta_0^{(i)})-\theta_{\mu}|<\varepsilon$  for every $x$
such that $\norm{ x } \le R'({T_\block}) + T_\block R$ (in fact, this holds for all $x$ within distance $R_\dens$ from the origin and $R_\dens\ge R'(T_\block)+T_\block R$). The sites, where the densities of $\eta^{(1)}, \eta^{(2)}$ are close to $\theta_{\mu}$ due to the well-startedness, encompass the entire $n=0$ (bottom) level of the space-time pyramid $\Psi_{T_\block}$, see
also Figure~\ref{fig:pyramid}. Due to this the event $\Psi_{T_\block}$ holds with high probability. Indeed, by defining the events $A_0 := \emptyset$ and
  \begin{equation*}
    A_j := \big\{ \exists z \in B_{R'(T_\block)+ (T_\block-j)R}(0) :
      |\delta_R(z;\eta_{j}^{(1)})-\theta_{ \mu} |>\varepsilon \big\},
  \end{equation*}
  we see that on the event $\{\Block(0,0)\text{ is well-started}\}$
  \begin{align*}
    \P \big( \Psi_{T_\block}^c \big | \mathcal{F}_0\big)
    &\le 2 \, \P \Bigg( \bigcup_{j=1}^{T_\block} A_j  \Bigg |  \mathcal{F}_0 \Bigg)
    \le 2 \sum_{j=1}^{T_\block} \P \big ( A_j \cap A_{j-1}^c \big | \mathcal{F}_0\big)
    \le 2 \sum_{j=1}^{T_\block} \P ( A_j \mid  A_{j-1}^c  , \mathcal{F}_0).
  \end{align*}
  Together with Lemma~\ref{lem:dens.dev} it follows with some constants $c_1,c_2>0$ that
  \begin{equation}
  \label{eqn:rhs_psi}
   \ind_{\{\Block(0,0)\text{ is well-started}\}} \P \big( \Psi_{T_\block}^c \big | \mathcal{F}_0\big) 
   \le c_1 T_\block (L_\block)^d \exp(-c_2 \, V_R^d).
  \end{equation}
  
 In order to utilise the control guaranteed by the pyramids $\Psi_n$ we introduce events that describe properties (ii) and (iii) in Definition~\ref{def:good_block}:
  \begin{align*}
    \mathcal{C}
    &:= \big\{ \eta_{T_\block}^{(1)}(x) = \eta_{T_\block}^{(2)}(x)
      \text{ for } \norm{x} \le 3 R'_\block \big\} \\
    \mathcal{D} &:= \big\{ \big(\eta^{(1)}_{T_\block},
        \eta^{(2)}_{T_\block}\big)
      \text{ satisfy \eqref{eq:well-started2} around }
      L'_\block e \text{ for all } \norm{e}\le 1 \big\}.
  \end{align*}
  We are interested in the conditional probability $\P(\mathcal{C} \cap \mathcal{D} | \mathcal{F}_0)$ on the event that $\Block(0,0)$ is well-started. Clearly it holds that
    \begin{align}
    \label{eqn:lb_coupling1}
      \P \big( \mathcal{C}^c \cup \mathcal{D}^c \big| \mathcal{F}_0\big)
      \le \P \big( \mathcal{C}^c \cap \Psi_{T_\block} \big| \mathcal{F}_0\big)
      + \P\big( \Psi_{T_\block}^c  \big | \mathcal{F}_0\big) + \P \big( \mathcal{D}^c \big | \mathcal{F}_0\big).
  \end{align}
  By \eqref{eqn:rhs_psi} the second term in \eqref{eqn:lb_coupling1} decays in $R$ for well-started configurations. To deal with the third term, note that it follows from Lemma~\ref{lem:dens.dev2} and Lemma~\ref{lem:dens.dev}
  that for some constants $c_3,c_4 >0$
  \begin{equation}\label{eqn:rhs_Y}
    \ind_{\{\Block(0,0)\text{ is well-started}\}} \P(\mathcal{D}^c | \mathcal{F}_0) \le c_3 T_\block (L_\block)^d \exp(-c_4 \, V_R^d).
  \end{equation}
  It remains to find a bound for
  $\P \big( \mathcal{C}^c \cap \Psi_{T_\block} \big | \mathcal{F}_0 \big)$. To this end fix $k \in \{ 1, \dots, T_\block \}$. By a union bound and Markov's inequality
  \begin{equation}
    \label{eqn:Xc}
    \begin{split}
      & \P\big(\{ \exists\, \abs{x} \le R'(k) \text{ such that }\eta^{(1)}_k(x) \neq \eta^{(2)}_k(x) \} \cap \Psi_k \big| \mathcal{F}_0\big) \\
      &\qquad \le \sum_{x \in B_{R'(k)}(0)}
      \E\big[ \ind_{\Psi_{k}} |\eta_{k}^{(1)}(x)-\eta_{k}^{(2)}(x)| \big| \mathcal{F}_0 \big]\\
      & \qquad = \E\Big[\sum_{x \in B_{R'(k)}(0)}\ind_{\Psi_{k-1}}
        \E\big[|\eta_{k}^{(1)}(x)-\eta_{k}^{(2)}(x)|\big | \mathcal{F}_{k-1}\big] \Big| \mathcal{F}_0\Big].
    \end{split}
  \end{equation}

  In light of the coupling \eqref{eqn:coupling}, we have
  \begin{align*}
    \E \big[ \abs{ \eta_{k}^{(1)}(x)- \eta_{k}^{(2)}(x) } \big|\,
      \mathcal{F}_{k-1} \big]
    & = \P \Big( U(x,k) \le
      \abs[\big]{\varphi_{\mu}(\delta_R(x; \eta_{k-1}^{(1)}))
        - \varphi_{\mu}(\delta_R(x; \eta_{k-1}^{(2)}))}
      \Big|\, \mathcal{F}_{k-1} \Big) \\
    & = \abs[\big]{\varphi_{\mu}(\delta_R(x; \eta_{k-1}^{(1)}))
      - \varphi_{\mu}(\delta_R(x; \eta_{k-1}^{(2)})) }.
  \end{align*}

  Now
  $\delta_R(x;\eta_{k-1}^{(i)}) \in [\theta_{\mu} -\varepsilon,\theta_{\mu}+ \varepsilon]$
  for $i = 1,2$ on the event $\Psi_{k-1}$ and by
  Lemma~\ref{lem:alpha_sequences}, $\varphi_{\mu}$ is a contraction with
  Lipschitz constant $\kappa(\mu, \varepsilon)<1$ on this interval.
  Therefore
  \begin{align*}
    \ind_{\Psi_{k-1}} \E \big[
      \abs{ \eta_{k}^{(1)}(x) - \eta_{k}^{(2)}(x) }
      \,\big|\, \mathcal{F}_{k-1} \big]
    \le & \ind_{\Psi_{k-1}} \kappa(\mu, \varepsilon) \abs[\big]{
      \delta_R(x; \eta_{k-1}^{(1)}) - \delta_R(x; \eta_{k-1}^{(2)}) } \\
    \le& \ind_{\Psi_{k-1}} \kappa(\mu,\varepsilon) V_R^{-d}
    \sum_{y \in B_R(x)}\abs[\big]{\eta_{k-1}^{(1)}(y)-\eta_{k-1}^{(2)}(y)}.
  \end{align*}
  Plugging this back into \eqref{eqn:Xc} yields
  \begin{align*}
    & \P\big(\{ \exists\, \abs{x} \le R'(k) \text{ such that }\eta^{(1)}_k(x) \neq \eta^{(2)}_k(x) \} \cap \Psi_k \big | \mathcal{F}_0\big) \\
    & \qquad \le \kappa(\mu,\varepsilon) V_R^{-d} \, \sum_{x \in B_{R'(k)}(0)}
    \sum_{y \in B_R(x)} \E
    \Big[ \ind_{\Psi_{k -1}}
      \abs[\big]{\eta_{k -1}^{(1)}(y)-\eta_{k -1}^{(2)}(y)} \Big | \mathcal{F}_0\Big].
  \end{align*}
  By inductively repeating this step another $k -1$ times, we can upper
  bound the right hand side of the last display by
  \begin{align*}
    & \kappa(\mu,\varepsilon)^{k} V_R^{-dk} \,
    \E \Big[ \sum_{x \in B_{R'(k)}(0)} \sum_{y_1 \in B_R(x)}
      \sum_{y_2 \ \in B_R(y_1)}\cdots
      \sum_{y_{k} \in B_R(y_{k-1})}
      \abs[\big]{\eta_{0}^{(1)}(y_{k})-\eta_{0}^{(2)}(y_{k})} \Big | \mathcal{F}_0\Big].
  \end{align*}
  Since $\abs[\big]{ \eta_0^{(1)}(y_k)-\eta_0^{(2)}(y_k) } \le 1$, with
  $k=T_\block$ we obtain
  \begin{align}
    \label{eqn:rhs_X}
     \ind_{\{\Block(0,0)\text{ is well-started}\}} \P\big( \mathcal{C}^c \cap \Psi_{T_\block}  \big| \mathcal{F}_0\big)
    &\le  \kappa(\mu,\varepsilon)^{T_\block} V_{R'(T_\block)}^d.
  \end{align}
  The choice $c_\ttime > -(d+1)/\log \kappa$
  guarantees that this probability tends to zero as $R$ goes to infinity.
  Combining \eqref{eqn:rhs_X} together with~\eqref{eqn:rhs_Y}
  and~\eqref{eqn:rhs_psi} gives that, on the event that $\Block(0,0)$ is well-started, all the terms on the right-hand side
  of \eqref{eqn:lb_coupling1} tend to zero as $R$ goes to infinity, thus proving \eqref{eq:welltogood_zero}.
\end{proof}

\subsection{Proof of Theorem~\ref{thm:complete_convergence}}
\label{sec:prfcompconv}

We now have all required tools to prove complete convergence of the BARW.
Given these tools, the proof is relatively standard and thus it is kept
brief.

\begin{proof}[Proof of Theorem~\ref{thm:complete_convergence}]
  As the Dirac measure concentrated around $\eta \equiv 0$ is an invariant
  distribution for $\eta$ we only need to show existence of a
  unique non-trivial limiting invariant measure which does not charge the
  empty configuration. To this end, let $\nu_0$ be the product measure on
  $\Z^d$ such that, for all $x \in \Z^d$, $\eta_0(x) = 1$ with probability
  $p >0$  and $\eta_0(x)=0$ otherwise. For any $n\ge 1$, denote by $\nu_n$ the
  distribution of $\eta_n$ given that $\eta_0$ is distributed as $\nu_0$.

  Since the set of all probability measures on $\{0,1\}^{\Z^d}$ is
  compact, there exists a subsequence along which
  $\tfrac{1}{N} \sum_{n=0}^N \nu_n$ converges to some probability measure
  $\nu$ on $\{0,1\}^{\Z^d}$. From a standard result for interacting
  particle systems, see e.g.~\cite[Proposition~1.8]{liggett1985interacting},
  any such subsequential limit $\nu$ must be invariant for the
  process $\eta$.

  To show that $\nu$ is non-trivial (and actually gives zero mass to the
    empty configuration $\eta \equiv 0$), it suffices to show that $\eta$
  survives almost surely. As we chose $\nu_0$ to be a product measure and
  since for any fixed $R$ the blocks defined in
  Section~\ref{sec:survival} depend only on finitely many sites, it
  follows that at time $0$ there are almost surely infinitely many
  \emph{well-started} blocks and hence by \eqref{eqn:wellgood} infinitely
  many \emph{good} blocks. By the correspondence of the blocks with
  supercritical oriented site percolation and the fact that supercritical
  oriented site percolation starting from infinitely many occupied sites
  does not die out (see e.g.~\cite[Theorem B24]{liggett1999stochastic}),
  we have $\P_{\nu_0}( \exists n \ge 1 : \eta_n \equiv 0 )=0$.

  Furthermore, the measure $\nu$ is extremal, because any limiting
  invariant distribution $\nu'$ which gives zero mass to $\eta \equiv 0$
  must be unique. Indeed, if two stationary distributions existed with
  this property, then by Theorem~\ref{thm:coupling} they would coincide
  on finite subsets of $\Z^d$, and would therefore be equal. Furthermore,
  under $\nu$, $\eta$ has exponentially decaying correlations in space
  and in time, see~\cite[Corollary~3.18]{DepperschmidtPhD}, which in particular implies ergodicity with respect to spatial
  shifts. Indeed, using the construction of good blocks from the proof of
  Theorem~\ref{thm:coupling} below, this can be deduced from the
  corresponding property of supercritical oriented percolation in a
  fairly straightforward way, see for example the analogous construction
  in \cite[Section~3.4]{DepperschmidtPhD} for the related model of a
  locally regulated population from \cite{birkner2007survival}.

  Finally, in order to verify the complete convergence, consider any
  (fixed) initial condition $\widetilde{\eta}_0\in \{0,1\}^{\Z^d}$, a
  finite box $B \subset \Z^d$ centred at the origin and a configuration
  $\zeta \in \{0,1\}^B$. With
  $\mathcal{S} := \{ \eta_m \not\equiv 0 \text{ for all } m \in \N\}$ we
  have to check that
  \begin{equation}
    \label{eq:compconv}
    \lim_{n\to\infty} \P_{\widetilde{\eta}_0}(\{ \eta_n|_B = \zeta\}
      \cap \mathcal{S})
    = \P_{\widetilde{\eta}_0}(\mathcal{S}) \nu(\{ \eta_0|_B = \zeta\}).
  \end{equation}
  Pick $\varepsilon>0$. The coupling construction from
  the proof of Theorem~\ref{thm:coupling} and standard properties of supercritical
  oriented percolation show that one can pick $L' \in \N$ and
  $T' \in \N$ large so that
  $\big| \P_{\eta_0'}(\{ \eta_m|_B = \zeta\}) - \nu(\{ \eta_0|_B =
      \zeta\}) \big| \le \varepsilon$ for all $m \ge T'$ and all starting
  configurations
  \begin{equation*}
    \eta_0' \in G' := \left\{ \widetilde{\eta} \in \{0,1\}^{\Z^d}  : \;
      \parbox{0.6\linewidth}{The density of well-started sub-boxes,
        where the local density of $\widetilde{\eta}$ satisfies \eqref{eq:well-started2}
        from Definition~\ref{def:well_started}, in a box of radius
        $L' R_\block$ is at least $1/2$.} \: \right\}.
  \end{equation*}
  Furthermore, since starting from any non-trivial initial condition
  there is a positive chance of producing a well-started box in a
  finite number of steps, a ``restart'' argument together with the
  construction from Theorem~\ref{thm:coupling} shows that
  $\P_{\widetilde{\eta}_0}(\mathcal{S} \Delta \{ \eta_n \in G'\}) \le
  \varepsilon$ for all large enough~$n$.
  Thus
  \begin{align*}
    & \left| \P_{\widetilde{\eta}_0}(\{ \eta_n|_B = \zeta\} \cap \mathcal{S})
    - \P_{\widetilde{\eta}_0}(\mathcal{S}) \nu(\{ \eta_0|_B = \zeta\})\right| \\
    & \quad \le
    \left| \P_{\widetilde{\eta}_0}(\{ \eta_n|_B = \zeta\} \cap \{ \eta_{n/2} \in G'\})
    - \P_{\widetilde{\eta}_0}(\{ \eta_{n/2} \in G'\}) \nu(\{ \eta_0|_B = \zeta\})\right|
    + 2 \varepsilon \\
    & \quad \le \E_{\widetilde{\eta}_0}\Big[ \ind_{\{ \eta_{n/2} \in G'\}}
      \big| \P_{\widetilde{\eta}_0}(\eta_n|_B = \zeta \,|\, \mathcal{F}_{n/2}) - \nu(\{ \eta_0|_B = \zeta\}) \big| \Big]
    + 2 \varepsilon \le 3 \varepsilon.
  \end{align*}
  Taking $n\to\infty$ and then $\varepsilon\downarrow0$ proves
  \eqref{eq:compconv}.
\end{proof}

\section{Extinction results}
\label{sec:extinction}

We provide here a simple proof of Theorem~\ref{thm:R_fixed_extinction}
describing the extinction regime.

\begin{proof}[Proof of Theorem~\ref{thm:R_fixed_extinction}]
  Let $R \in \N$ and $\mu>0$ be such that
  \begin{equation}
    \label{def:mutilde}
    \widetilde{\mu} := V_R^d \, \varphi_{\mu}\big(V_R^{-d}\big) = \mu e^{-\mu V_R^{-d}} < 1.
  \end{equation}
  Then $\psi(w) := \widetilde{\mu} w$ fulfils
  $\varphi_\mu(w) \le \psi(w)$ on $ [0, 1] \cap V_R^{-1} \Z$ (note
    that if $w \ge V_R^{-1}$, we have
    $\varphi_\mu(w) = \mu w \exp(-\mu w) \le \mu w \exp(-\mu
      V_R^{-d}) = \widetilde{\mu} w$ and $\varphi_\mu(0)=\psi(0)$).

  Thus, we can define a process $(\widetilde{\eta}_n)_{n \in \N_0}$ with
  $\widetilde{\eta}_0 = \eta_0$ using this $\psi$ as in
  \eqref{eqn:etatilde}. By the coupling construction from
  Section~\ref{ss:coupling} and specifically
  Lemma~\ref{lem:coupling_lemma}(b) we conclude that
  $\eta_n(x) \le \widetilde{\eta}_n(x)$ holds for all
  $n \in \N, x \in \Z^d$. Since $\psi$ is a linear function, we have
  \begin{align*}
    \E[\widetilde{\eta}_n(x)]
    & = \widetilde{\mu}V_R^{-d} \sum_{y \in B_R(x)}
    \E[\widetilde{\eta}_{n-1}(y)]
  \end{align*}
  Iterating this $n$ times shows
  \begin{align*}
    \E[\widetilde{\eta}_n(x)]
    & = \widetilde{\mu}^n \sum_{z \in \Z^d} p^{(n)}(x,z)
    \E[\eta_0(z)] \le \widetilde{\mu}^n
  \end{align*}
  where $p^{(n)}$ is the $n$-fold convolution of the uniform transition
  kernel on $B_R(0)$ with itself. Since $\widetilde{\mu} < 1$ this
  combined with the coupling shows that
  $\sum_{n=1}^\infty \P(\eta_n(x) > 0) < \infty$ so that indeed for every
  $x \in \Z^d$
  \begin{align*}
    \P(\eta_n(x) = 0 \text{ for all $n$ large enough}) = 1.
  \end{align*}

  Next note that the equation $\mu \exp(-\mu V_R^{-d})=1$ , i.e.\ the
  equality in \eqref{def:mutilde}, has two positive real solutions
  $\mu_1, \mu_2$ such that $1 < \mu_1 < \mu_2 < \infty$ when $R \ge 1$
  (when $R=0$ there is always extinction). The function
  $\mu \mapsto \mu \exp (-\mu V_R^{-d})$ is unimodal and vanishes at 0 as
  well as at $+\infty$, so if $\mu < \mu_1$ or $\mu > \mu_2$ there is
  extinction.

  We can rewrite $\mu \exp ( -\mu V_R^{-d} )=1$ as $y e^y =x$ where
  $y = -\mu V_R^{-d}$ and $x = -V_R^{-d}$. When $x \in [-1/e, 0)$, this
  equation has two real solutions $y_1 =W_0(x)$ and $y_2 =W_{-1}(x)$,
  where $W_0$ and $W_{-1}$ are two branches of the Lambert $W$ function.
  Since $\mu = -V_R^d y$, the two solutions of
  $\mu \exp ( -\mu V_R^{-d} )=1$ are
  \begin{equation*}
    \mu_1 = -V_R^d W_0 ( -V_R^{-d}), \qquad
    \mu_2 = -V_R^d W_{-1} ( -V_R^{-d} ).
  \end{equation*}

  Since $x \in [-1/e, 0)$, we can express $W_0(x)$ with its Taylor series
  centred at 0, which has radius of convergence $1/e$, that is
  \begin{equation*}
    W_0(x) = \sum_{n=1}^{\infty} \frac{(-n)^{n-1}}{n!} x^n
    = x-x^2 +\frac{3}{2} x^3 -\frac{8}{3} x^4 +
    \cdots
  \end{equation*}
  This gives
  \begin{equation}
    \label{eq:mu1asympt}
    \mu_1 = -V_R^d W_0 ( -V_R^{-d} )
    = 1 + V_R^{-d} + \frac{3}{2} V_R^{-2d} + \cdots
  \end{equation}
  For the second solution, we use that
  \begin{equation*}
    -1-\sqrt{2u} -u < W_{-1}(-e^{-u-1} ) < -1-\sqrt{2u} - \frac{2u}{3}
  \end{equation*}
  for every $u >0$. Take $u = d \log V_R-1$. Then the formula above gives
  \begin{align}
    \label{eq:mu2asympt}
    - \sqrt{2d \log V_R -2} - d \log V_R
    & < W_{-1}( -V_R^{-d}) \notag \\
    & <  -\frac{1}{3} - \sqrt{2d \log V_R -2} - \frac{2d}{3} \log V_R,
  \end{align}
  which gives the result.
\end{proof}

\section{Auxiliary results}

We prove here the auxiliary technical results that were omitted in the
previous sections. Section~\ref{sec:pf.lem:f1dprofilelbd} deals with
Lemma~\ref{lem:f1dprofilelbd} which was used in the construction of the
comparison density profiles $\xi_n^-$ in Section~\ref{sec:survival}. In
Section~\ref{sec:construction}, we then show
Lemma~\ref{lem:alpha_sequences} used in the proof of the complete
convergence in Section~\ref{sec:complete_convergence}.  Finally, in
Section~\ref{sec:appendix_2} we provide a proof of
Proposition~\ref{prop:cml}.

\subsection{Proof of Lemma~\ref{lem:f1dprofilelbd}}
\label{sec:pf.lem:f1dprofilelbd}

Recall that $f:\mathbb Z\to [0,1]$ is defined by
\begin{equation*}
  f(x) = \min\big\{ (\varepsilon_0 +  x/\lceil wR \rceil) \ind_{x \ge 0}, 1 \big\}.
\end{equation*}
It is immediate that this function satisfies the properties in
\eqref{eqn:f_properties}. Therefore, it only remains to show that
\eqref{eqn:density_domination} holds for a suitable choice of parameters.

It is clear that the larger the growth factor $a$ is, the easier it is
for \eqref{eqn:density_domination} to be satisfied. Setting
$\varepsilon_0 = \min\{(a-1)^2,1/100\}$, it follows that
\begin{equation*}
  a\delta_R(x;f) \ge (1+\sqrt{\varepsilon_0})\delta_R(x;f) \quad
  \text{ for all } x \in \Z,
\end{equation*}
which lets us reduce to the case where $1<a<11/10$ and
$a = 1 + \sqrt{\varepsilon_0}$.

We now set
\begin{equation}
  \label{eqn:awe}
  w = 1/\sqrt{\varepsilon_0},
\end{equation}
and define
\begin{equation*}
  C_0 := \{ y \in \Z : y < 0 \}, \qquad C_1 := \{ y \in \Z : y \ge \lceil wR
    \rceil\},
\end{equation*}
so that $f(y)=0$ for every $y \in C_0$ and $f(y)=1$ for every $y \in C_1$.
Since $w>1$, exactly one of the two sets $B_R(x) \cap C_0$
and $B_R(x)\cap C_1$ can be non-empty. Clearly
\eqref{eqn:density_domination} holds when $B_R(x) \subseteq C_0$, or
$B_R(x) \subseteq C_1$.

When $B_R(x) \cap (C_0 \cup C_1) = \emptyset$,  then
$f(y)=\varepsilon_0 + y/\lceil wR \rceil$ for every $y \in B_R(x)$ and
thus $\delta_R(x;f) = f(x)$, so \eqref{eqn:density_domination} holds as
well.

The remaining two cases are more delicate. When
$B_R(x) \cap C_0 \neq \emptyset$ and $B_R(x) \not\subseteq C_0$, that is
when $-R \le x < R$, then the density of $f$ around $x$ can be written as
\begin{align*}
  \delta_R(x;f) =  V_R^{-1} \sum_{y=0}^{x+R} f(y)
  &=  V_R^{-1} \sum_{y=0}^{x+R} \Big(\varepsilon_0+\frac{y}{\lceil wR \rceil}\Big) \\
  &= V_R^{-1}\Big(  (x+R+1)\varepsilon_0 +\frac{1}{2\lceil wR \rceil}(x+R)(x+R+1)\Big).
\end{align*}
Using this, \eqref{eqn:density_domination} is equivalent to
\begin{align*}
  a(x+R+1)\varepsilon_0 + \frac{a}{2\lceil wR \rceil}(x+R)(x+R+1) &\ge V_R f(x+\lceil sR \rceil) \\
  &= V_R \Big( \varepsilon_0 +\frac{x}{ \lceil wR \rceil} +\frac{\lceil sR
      \rceil}{\lceil wR  \rceil }\Big).
\end{align*}
Rearranging terms, we arrive at a quadratic inequality
\begin{equation}
  \label{eqn:quadratic}
  \alpha x^2 + \beta x + \gamma \ge 0,
\end{equation}
where
\begin{align*}
  \alpha &= \frac{a}{2\lceil wR \rceil }, \\
  \beta&= a \varepsilon_0 + \frac{R}{\lceil wR \rceil}
  \Big( \frac{a}{2}-1 \Big) \Big( 2 + \frac{1}{R} \Big), \\
  \gamma &= a\varepsilon_0R \Big( 1+\frac{1}{R} \Big)
  +  \frac{aR^2}{2\lceil wR \rceil} \Big( 1+\frac{1}{R}\Big) -R \Big(2+\frac{1}{R} \Big)
  \Big( \varepsilon_0+\frac{\lceil s R\rceil }{\lceil wR \rceil} \Big).
\end{align*}
As $\alpha>0$ for our choices of parameters, \eqref{eqn:quadratic} and
hence \eqref{eqn:density_domination} follow immediately if the polynomial
$\alpha x^2 + \beta x+\gamma$ has no real roots.
The discriminant of \eqref{eqn:quadratic} is given by
\begin{equation}
  \label{eqn:quad2}
  \begin{split}
    \beta^2 - 4 \alpha \gamma
    & = \Big( a \varepsilon_0 + \frac{2}{w}
      \Big( \frac{a}{2}-1 \Big) \Big)^2 - \frac{2a}{w} \Big( a \varepsilon_0
      +  \frac{a}{2w}-2 \Big( \varepsilon_0+\frac{ s }{w} \Big)  \Big)+O(R^{-1}) \\
    & = (a \varepsilon_0)^2 + \frac{4}{w^2}(1-a+as)+O(R^{-1}).
  \end{split}
\end{equation}
We now choose
\begin{equation*}
  s= \frac{\sqrt{\varepsilon_0}}{1+\sqrt{\varepsilon_0}}- \varepsilon_0,
\end{equation*}
which is clearly positive for $\varepsilon_0\in (0,1/100)$. Recalling
also \eqref{eqn:awe} and that $a = 1+\sqrt {\varepsilon_0}$, the right-hand
side of \eqref{eqn:quad2} (without the error term) equals
\begin{align*}
  \varepsilon_0^2(\sqrt {\varepsilon_0} - 3)(1+\sqrt{\varepsilon_0})
\end{align*}
which is clearly negative. As consequence, the quadratic inequality
\eqref{eqn:quadratic} holds for all $R$ big enough, depending only on
$a$, and thus \eqref{eqn:density_domination} holds also in this case.

For the final case, when  $B_R(x)\cap C_1 \neq \emptyset$ that is
$\lceil wR \rceil - R \le x \le \lceil wR \rceil - R$, we observe that
the right-hand side of \eqref{eqn:density_domination} is bounded by one
and the left-hand side is increasing in $x$. It is thus sufficient to
show that $a \delta_R(\lceil wR \rceil - R-1) \ge 1$. Using again the
fact that $f$ is linear in the $R$-neighbourhood of
$\lceil wR \rceil - R-1$, this is equivalent to showing
$a f(\lceil wR \rceil - R-1) \ge 1$. Recalling the definitions of $a$, $w$
and $s$ in terms of $\varepsilon_0$, we have
\begin{align*}
  a f(\lceil wR \rceil - R-1)
  &= a \Big(\varepsilon_0  + \frac{\lceil wR \rceil -R -1}{\lceil wR
      \rceil}\Big)
  \\&= a \big(\varepsilon_0 + 1- w^{-1} + O(R^{-1})\big)
  \\&= (1+\sqrt {\varepsilon_0}) \big(\varepsilon_0 + 1-
    \sqrt{\varepsilon_0} + O(R^{-1})\big)
  \\&= 1+\varepsilon_0^{3/2} + O(R^{-1}),
\end{align*}
and thus the required inequality is satisfied for $R$ large enough.
\hfill $\qed$

\subsection{Proof of Lemma~\ref{lem:alpha_sequences}}
\label{sec:construction}

We now prove Lemma~\ref{lem:alpha_sequences}, exploiting 
properties of $\varphi_\mu $ in the vicinity of its fixpoint
$\theta_\mu $.

\begin{proof}[Proof of Lemma~\ref{lem:alpha_sequences}]
  To prove that $\varphi_\mu $ is a contraction in the vicinity of its
  critical point $\theta_\mu  = \mu^{-1}\log \mu$, it suffices to
  observe that $\abs{\varphi_\mu '(w)} < 1$ in some neighbourhood of
  $\theta_\mu $. Since $\varphi_\mu '(w) = \mu  e^{-\mu  w} (1-\mu w)$,
  it holds that
  $\abs{\varphi'_\mu (\theta_\mu )} = \abs{1 - \log \mu}< 1$ if
  $\mu  \in (1,e^2)$.  The statement then follows by the continuity of
  the derivative.

  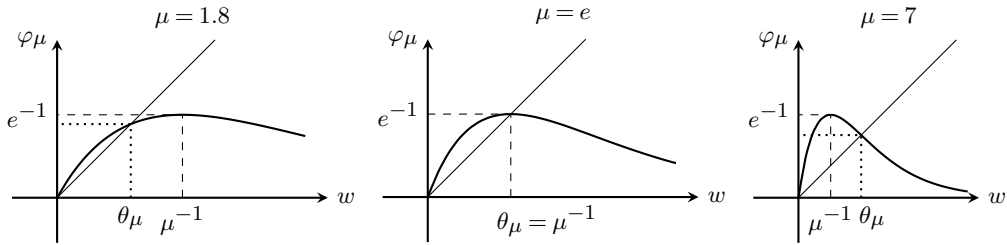
\begin{figure}[b!]
    \begin{tikzpicture}[scale=3]
      \draw (0.6,0.8) node {$\mu=1.8$};
      \draw[thick, -stealth] (-0.2, 0) -- (1.2, 0) node[right] {$w$};
      \draw[thick, -stealth] (0, -0.2) -- (0, 0.7) node[left] {$\varphi_\mu$};
      \draw (0,0) -- (0.7,0.7);
      \draw[domain=0:1.1, smooth, variable=\x, black, thick=1] plot ({\x}, {1.8*\x*exp(-1.8*\x)});
      \draw[dashed] ({1/1.8},0) -- ({1/1.8},{exp(-1)}) -- (0, {exp(-1)});
      \draw (0, {exp(-1)}) node[left] {$e^{-1}$};
      \draw ({1/1.8},0) node[below] {$\mu^{-1}$};
      \draw[thick, dotted] ({ln(1.8)/1.8},0) -- ({ln(1.8)/1.8},{ln(1.8)/1.8}) -- (0,{ln(1.8)/1.8});
      \draw ({ln(1.8)/1.8},0) node[below] {$\theta_\mu$} ;
    \end{tikzpicture}
    \begin{tikzpicture}[scale=3]
      \draw (0.6,0.8) node {$\mu=e$};
      \draw[thick, -stealth] (-0.2, 0) -- (1.2, 0) node[right] {$w$};
      \draw[thick, -stealth] (0, -0.2) -- (0, 0.7) node[left] {$\varphi_\mu$};
      \draw (0,0) -- (0.7,0.7);
      \draw[domain=0:1.1, smooth, variable=\x, black, thick=1] plot ({\x}, {exp(1)*\x*exp(-exp(1)*\x)});
      \draw[dashed] ({exp(-1)},0) -- ({exp(-1)},{exp(-1)}) -- (0, {exp(-1)});
      \draw (0, {exp(-1)}) node[left] {$e^{-1}$};
      \draw ({exp(-1)-0.1},0) node[below right] {$\theta_\mu = \mu^{-1}$};
    \end{tikzpicture}
    \begin{tikzpicture}[scale=3]
      \draw (0.4,0.8) node {$\mu=7$};
      \draw[thick, -stealth] (-0.2, 0) -- (0.8, 0) node[right] {$w$};
      \draw[thick, -stealth] (0, -0.2) -- (0, 0.7) node[left] {$\varphi_\mu$};
      \draw (0,0) -- (0.7,0.7);
      \draw[domain=0:0.75, smooth, variable=\x, black, thick=1] plot ({\x}, {7*\x*exp(-7*\x)});
      \draw[dashed] ({1/7},0) -- (({1/7},{exp(-1)}) -- (0, {exp(-1)});
      \draw (0, {exp(-1)}) node[left] {$e^{-1}$};
      \draw ({1/7},0) node[below] {$\mu^{-1}$};
      \draw[thick, dotted] ({ln(7)/7},0) -- ({ln(7)/7},{ln(7)/7}) -- (0,{ln(7)/7});
      \draw ({ln(7)/7 +0.05},0) node[below] {$\theta_\mu$};
    \end{tikzpicture}
    \caption{The function $\varphi_\mu$, its fixpoint $\theta_\mu$ and its maximum in the case
      $\mu < e$ (left), $\mu = e$ (middle) and $e < \mu < e^2$ (right)}
    \label{fig:3cases}
  \end{figure}

  To find the sequences $\alpha_m$ and $\beta_m$, note first that
  $\varphi_\mu $ is increasing on $[0,1/\mu ]$ and decreasing on
  $[1/\mu ,\infty]$. It is convenient to consider three cases (cf.\ also Figure~\ref{fig:3cases}):
  \smallskip

  (1) If $\mu \in (1,e)$, then $\theta_{\mu} < 1/e < 1/\mu $, and thus
  $\varphi_{\mu}$ is a strictly increasing on $[0, 1/e]\ni \theta_\mu $,
  and $\varphi_{\mu}(w)>w$ if $w < \theta_{\mu}$, and
  $\varphi_{\mu}(w) < w$ when $w \in ( \theta_{\mu}, 1/e]$. Pick
  $\alpha_1 < \theta_{\mu}$ and $\beta_1 > 1/\mu$ satisfying
  $\varphi_{\mu}(\beta_1) \ge \varphi_{\mu}(\alpha_1)$. Put
  $\alpha_2 = (\alpha_1+ \varphi_{\mu}(\alpha_1))/2$,
  $\beta_2 =(e^{-1}+\mu^{-1})/2$, then we have indeed
  $\varphi_\mu([\alpha_1, \beta_1]) \subseteq(\alpha_2,\beta_2)$. From
  here on, we can simply iterate by setting
  \begin{equation}
    \label{eq:alpha.iter}
    \alpha_{m+1} = \frac{\alpha_m + \varphi_{\mu}(\alpha_m)}{2},
    \qquad \beta_{m+1} = \frac{\beta_m + \varphi_{\mu}(\beta_m)}{2}, \quad m \ge 2.
  \end{equation}
  This defines two sequences converging to $\theta_{\mu}$. Furthermore
  $\alpha_{m} < \alpha_{m+1} < \varphi_{\mu}(\alpha_m)$ and
  $\varphi_{\mu}(\beta_m) < \beta_{m+1} < \beta_m$, so
  $(\alpha_m)_{m \ge 1}$ is strictly increasing and $(\beta_m)_{m \ge 1}$
  is strictly decreasing. Since $\varphi_{\mu}$ is strictly increasing on
  $[0,1/e]$ we also have
  $\varphi_{\mu}([ \alpha_m, \beta_m]) \subseteq(\alpha_{m+1}, \beta_{m+1})$
  for every  $m \ge 0$, as required. \smallskip

  (2) Consider now the case $\mu = e$, that is when $\theta_\mu=1/e$ and
  $\varphi_{\mu}'(\theta_{\mu})=0$. Pick any $\alpha_1 < 1/e$ and
  $\beta_1 = 1/e$ such that
  $\varphi_{\mu}(\beta_1) \ge \varphi_{\mu}(\alpha_1)$; then build the
  sequence $(\alpha_m)_{m \ge 1}$ in the same way as in the case
  $\mu \in (1,e)$, i.e.\ as in \eqref{eq:alpha.iter}, using $m \ge 1$
  there. By construction, this sequence is strictly increasing and
  converges to $\theta_{\mu}$. For every $m \ge 2$,  let $\beta_m$ be the
  largest solution of $\varphi_{\mu}(x)= \varphi_{\mu}(\alpha_m)$. Since
  $w \mapsto \varphi_{\mu}(w)$ is (strictly) increasing if and only if
  $w \in [0,1/e]$, this defines a strictly decreasing sequence
  $(\beta_m)_{m \ge 1}$ converging to $\theta_{\mu}$ and such that
  \begin{equation*}
    \varphi_{\mu}( [\alpha_m, \beta_m])
    \subseteq[ \varphi_{\mu}(\alpha_m), 1/e] \subseteq(\alpha_{m+1},
      \beta_{m+1}),
  \end{equation*}
  as required.

  \smallskip

  (3) Finally, let $\mu \in (e, e^2)$, which implies $1/\mu < \theta_\mu$
  and $\varphi_\mu'(\theta_\mu) \in (-1,0)$. For the initial piece, pick
  $\alpha_1 < 1/\mu$ and $\lambda >0$ so small that
  $\mu^{-1} + \lambda e^{-1} < \theta_\mu$. Define, similarly to
  \eqref{eq:alpha.iter},
  \begin{equation*}
    \alpha_{m+1} = \lambda \varphi_{\mu}(\alpha_m) + (1-\lambda) \alpha_m , \; m \le m_0 - 1,
  \end{equation*}
  where $m_0$ is the smallest integer satisfying $\alpha_{m_0} > 1/\mu$.
  Note that by construction and the choice of $\lambda$, since
  $\varphi_\mu$ is strictly increasing on $[0,1/\mu]$ and bounded by $1/e$,
  we have
  $\alpha_1 < \alpha_2 < \cdots < \alpha_{m_0-1} \le 1/\mu < \alpha_{m_0} < \theta_\mu$.
  Choose $\beta_1 > \beta_2 > \cdots > \beta_{m_0} > 1/e$ ($> \theta_\mu$)
  so that $\varphi_\mu(\beta_m) > \varphi_\mu(\alpha_m)$ for
  $m=1,\dots,m_0$, then we have
  $\varphi_\mu([\alpha_m,\beta_m]) \subseteq(\alpha_{m+1},\beta_{m+1})$
  for $m=1,\dots,m_0-1$.

  Since $\alpha_{m_0} > 1/\mu$, the iteration has reached the
  decreasing part of $\varphi_\mu$ after $m_0$ steps and we thus must
  swap the roles of the upper and the lower boundary in each step: Set
  for $m \ge m_0$
  \begin{equation*}
    \alpha_{m+1} = \frac{\varphi_{\mu}(\beta_m)+\alpha_m}{2},
    \qquad \beta_{m+1} = \frac{\varphi_{\mu}(\alpha_m)+ \beta_m}{2}.
  \end{equation*}
  We note that if $\varphi_{\mu}(\alpha_m) < \beta_m$ and
  $\varphi_{\mu}(\beta_m) > \alpha_m$ then the same holds for
  $\alpha_{m+1}$ and $\beta_{m+1}$. Indeed
  $\varphi_{\mu}(\beta_m) > \alpha_{m}$ implies that
  $\alpha_{m+1} > \alpha_m$ and since $\varphi_{\mu}$ is decreasing then
  $\varphi_{\mu}(\alpha_{m+1}) < \varphi_{\mu}(\alpha_m)$. Similarly
  $\varphi_{\mu}(\alpha_m) < \beta_m$ implies that
  $\beta_{m+1} < \beta_m$ and so
  $\varphi_{\mu}(\alpha_m) = 2 \beta_{m+1} - \beta_m < \beta_{m+1}$.
  Combining the two gives
  $\varphi_{\mu}(\alpha_{m+1}) < \varphi_{\mu}(\alpha_m) < \beta_{m+1}$.
  In the same way we can prove that
  $\varphi_{\mu}(\beta_{m+1}) > \alpha_{m+1}$. Hence for $m \ge m_0$
  \begin{equation*}
    \varphi_{\mu}( (\alpha_m, \beta_m))
    \subseteq [ \varphi_{\mu}(\beta_m), \varphi_{\mu}(\alpha_m)]
    \subseteq(\alpha_{m+1}, \beta_{m+1}).
  \end{equation*}

  It is clear from the construction that in each one of the three cases
  $\alpha_1$ can be chosen arbitrarily small and $\beta_1 > 1/e$ (if a
    large $\beta_1$ is required, this can be achieved by decreasing
    $\alpha_1$ appropriately).
\end{proof}

\subsection{Proof of Proposition~\ref{prop:cml}}
\label{sec:appendix_2}

Note again that since $\max_{w \ge 0} \varphi_{\mu}(w)=1/e$, for every
initial condition $\Xi_0 \in \R_+^{\Z^d}$ of the coupled map lattice
defined in \eqref{eqn:cmlattice} we have $\Xi_1 \in [0, 1/e]^{\Z^d}$.
Thus we can assume without loss of generality that
$0 \le \Xi_0(z) \le 1/e$ for every $z \in \Z^d$. Assume moreover that
$\Xi_0(z_0)>0$ for some $z_0 \in \Z^d$, as it otherwise obviously holds
that $\Xi_n \equiv 0$ for all $n$. The proof follows ideas from Section 4 in~\cite{birkner2007survival}.

\begin{proof}[Proof of Proposition~\ref{prop:cml}]
  Fix $\varepsilon >0$ and let $a > 1 $ and $b > 0$ be such that
  $\psi(w)= aw \wedge b$ satisfies $\varphi_{\mu}(w) \ge \psi(w)$ for
  every $w \in [0,1]$. Since $\theta_{\mu}$ is a stable fixpoint when
  $\mu \in (1, e^2)$, we can choose sequences $(\alpha_m)_{m \ge 0}$,
  $(\beta_m)_{m \ge 0}$ as in Lemma~\ref{lem:alpha_sequences} with
  $\alpha_1 < b/2$ and a suitable $\beta_1 > 1/e$, such that
  $\varphi_{\mu}([\alpha_m,\beta_m]) \subseteq(\alpha_{m+1},\beta_{m+1})$
  and $\beta_{m^*}-\alpha_{m^*} < \varepsilon$ for some $m^* \in \N$.

  For a fixed $z \in \Z^d$ we show that there exists $n_0 > m^*$ such
  that $\Xi_n(z) \in [\alpha_{m^*}, \beta_{m^*}]$ for all $n \ge n_0$.
  We start by showing that
  \begin{align}
    \label{eqn:cml_lb}
    \Xi_n(z) \ge  \sum_{y \in \Z^d} p^{(n)}(z,y) \Big[ \big( a^n \Xi_0(y)\big) \wedge b \Big],
  \end{align}
  where $p^{(n)}(\cdot,\cdot)$ are the $n$-step transition probabilities
  of a random walk whose steps are uniformly distributed in
  $B_R(0) \cap \Z^d$. We can check \eqref{eqn:cml_lb} by induction. Using
  Jensen's inequality, it holds that
  \begin{align*}
    \Xi_{n+1}(z) = \varphi_{\mu}( \delta_R(z; \Xi_n))
    & \ge \psi \bigg( V_R^{-d} \, \sum_{x \in B_R(0)}
      \Xi_n(z+x) \bigg) \notag \\
    & \ge V_R^{-d} \, \sum_{x \in B_R(0)}
    \psi \big( \Xi_n(z+x)\big).
  \end{align*}
  Using the inductive assumption,
  \begin{align*}
    \psi \big( \Xi_n(z+x)\big) & \ge \bigg[ a \sum_{y \in \Z^d} p^{(n)}(z+x,y) \Big( \big( a^n \Xi_0(y)\big) \wedge b \Big) \bigg]  \wedge b\\
    & = \sum_{y \in \Z^d} p^{(n)}(z+x,y) \Big( \big( a^{n+1} \Xi_0(y)\big) \wedge ab \Big) \wedge b \\
    & \ge \sum_{y \in \Z^d} p^{(n)}(z+x,y) \Big( \big( a^{n+1} \Xi_0(y) \big) \wedge b \Big) \wedge b \\
    & =  \sum_{y \in \Z^d} p^{(n)}(z+x,y) \Big( \big( a^{n+1} \Xi_0(y)\big) \wedge b \Big),
  \end{align*}
  so
  \begin{equation*}
    \Xi_{n+1}(z) \ge V_R^{-d} \, \sum_{x \in B_R(0)}
    \sum_{y \in \Z^d} p^{(n)}(z+x,y)
    \Big( \big( a^{n+1} \Xi_0(y)\big) \wedge b \Big)
  \end{equation*}
  and the conclusion follows from the fact that
  \begin{equation*}
    V_R^{-d}  \sum_{x \in B_R(0)} p^{(n)}(z+x, y)
    = \sum_{x \in B_R(0)} p(z, z+x) p^{(n)}(z+x, y) = p^{(n+1)}(z,y).
  \end{equation*}

  For our fixed choice of $z$, we show that
  \begin{equation}\label{eqn:iter_0}
    \Xi_n(x)\in [\alpha_1, \beta_1] \text{ for all } n \ge n_0 \text{ and }
    \norm{x- z} \le 2 Rm^*.
  \end{equation}

  Take
  $n_1>\big( 4Rm^*+2  \norm{ z-z_0 } \big)^2 \vee \big(\big(\ln(b)-\ln(\Xi_0(z_0))\big)/\ln(a)\big)$
  large enough. By a local central limit theorem for symmetric finite
  range random walks, cf.\ \cite[Theorem 2.1.1]{lawler2010random} there
  exists $c>0$ such that $p^{(n_1)}(y,z_0) \ge c n_1^{-d/2}$ if
  $\norm{  y-z_0 } \le \sqrt{n_1}$. By letting
  $n_1> \big(\ln(b)-\ln(\Xi_0(z_0))\big)/\ln(a)$ it holds that
  $a^{n_1}\Xi_0(z_0)\wedge b = b$ and hence it follows with
  \eqref{eqn:cml_lb} that
  \begin{equation*}
    \Xi_{n_1}(y) \ge \sum_{w \in \Z^d} p^{(n_1)}(y,w)
    \Big[ \big( a^{n_1} \Xi_0(w)\big) \wedge b \Big]
    \ge p^{(n_1)}(y, z_0) \Big[ (a^{n_1} \Xi_0(z_0)) \wedge b \Big]
    \ge c n_1^{-d/2} b.
  \end{equation*}
  Using \eqref{eqn:cml_lb} again, we deduce that for any $n_2<\sqrt{n_1}/2$
  \begin{align*}
    \Xi_{n_1+n_2}(x) & \ge \sum_{y \in \Z^d} p^{(n_2)}(x,y) \Big[ (a^{n_2} \Xi_{n_1}(y)) \wedge b \Big] \\
    & \ge \sum_{y \in B_{\sqrt{n_1}}(z_0)} p^{(n_2)}(x,y) \Big( \big( a^{n_2} c n_1^{-d/2} \big) \wedge 1 \Big)b.
  \end{align*}
  Choosing $n_2 = d \log n_1 - 2 \log c$ gives that
  $\big( a^{n_2} c n_1^{-d/2} \big) \wedge 1  = 1$ and, since
  $B_{n_2}(x) \subseteq B_{\sqrt{n_1}}(z_0)$ when
  $n_1 > \big( 4Rm^*+2  \norm{ z-z_0 } \big)^2$, the above is larger than $b$.

  Since $b > 2 \alpha_1$ and trivially $\varphi_{\mu}(w)\le 1/e < \beta_1$
  for every $w \ge 0$, this shows \eqref{eqn:iter_0}.
  It follows that
  \begin{equation*}
    \Xi_{n+1}(x) = \varphi_{\mu}(\delta_R(x; \Xi_n)) \in [\alpha_2, \beta_2]
    \text{ for all } n \ge n_0 \text{ and } \norm{ x-z } \le (2m^*-1)R
  \end{equation*}
  and iterating $m^*$ steps shows that
  \begin{equation*}
    \Xi_{n+m^*-1}(x) \in [\alpha_m,\beta_m] \text{ for all }
    n \ge n_0 \text{ and } \norm{x-z} \le m^* R.
  \end{equation*}
  Take $x=z$ to conclude that $\Xi_n(z) \in [\alpha_m, \beta_m]$ for
  $n \ge n_0+m^*$.
\end{proof}

\section{Open Questions}
\label{sec:questions}

We collect here some natural follow-up questions to our results, several
of them were already mentioned in the text.

\begin{figure}[b!]
    \centering
    \includegraphics[width = 0.48\linewidth]{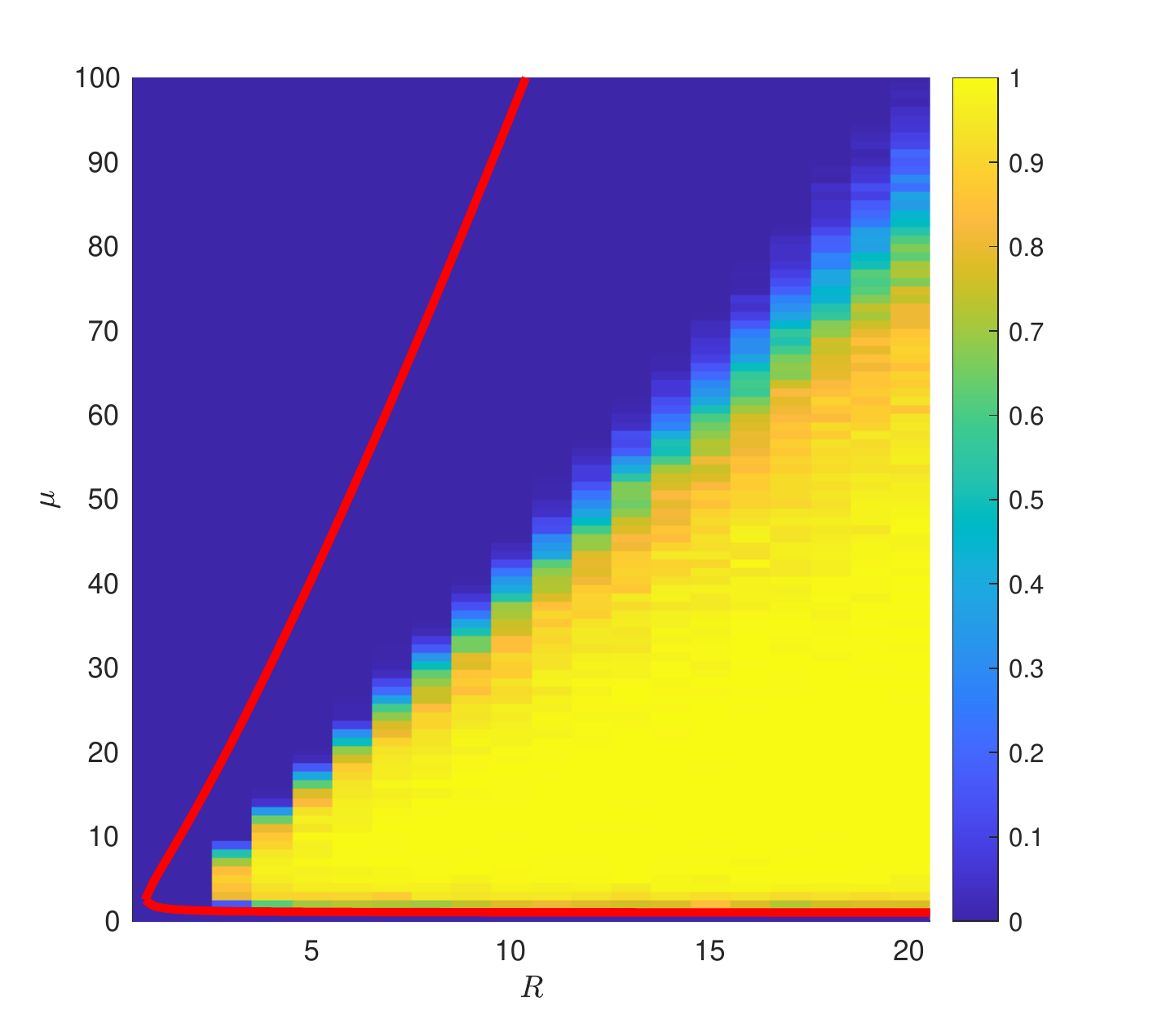}
    \includegraphics[width = 0.48\linewidth]{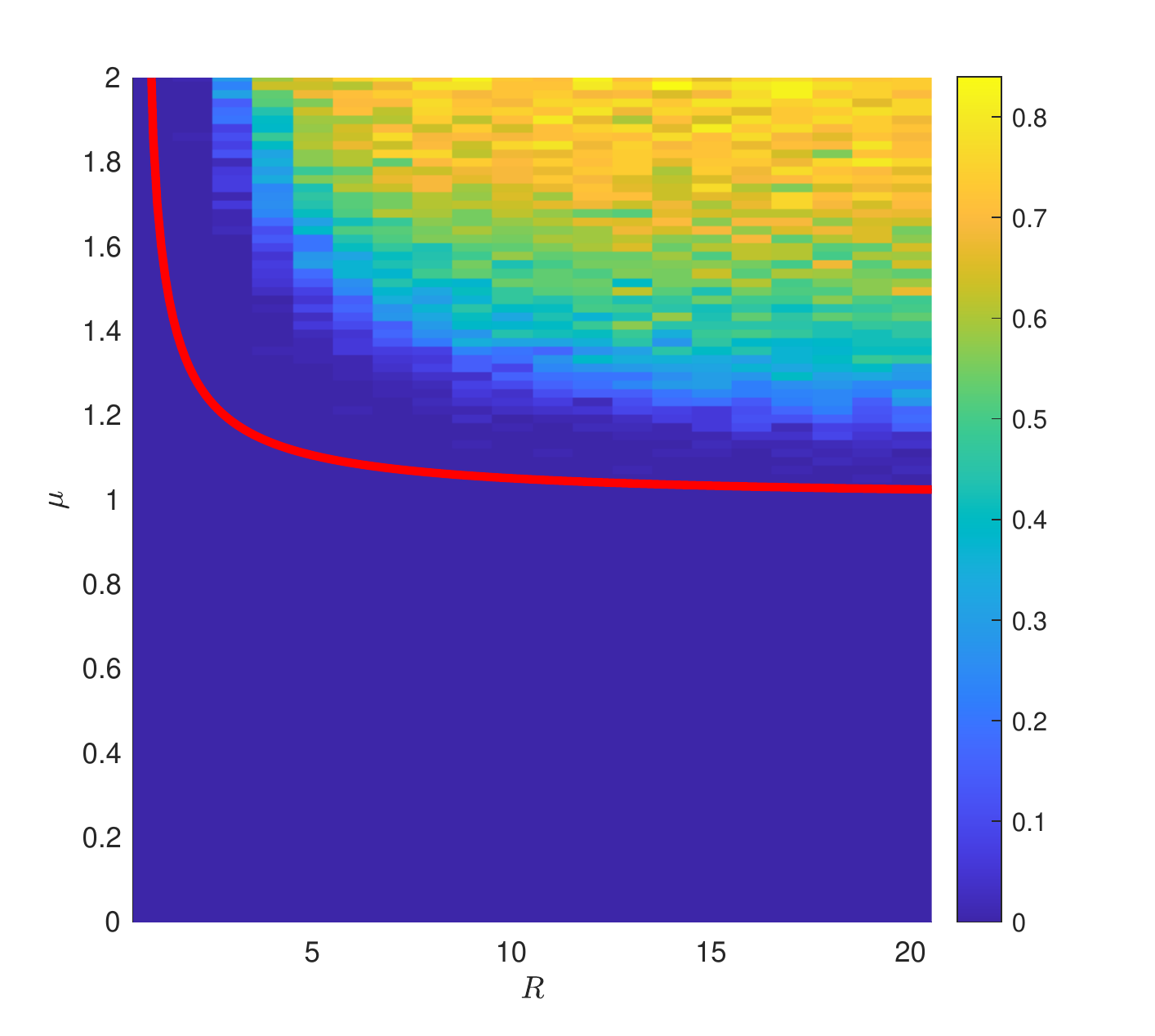}
    \caption{Simulations of the ``phase diagram'' for a one-dimensional
      BARW on $\Z/1000\Z$ with initial condition $\eta_0 = \delta_0$,
      showing a Monte Carlo estimate of the survival probability as a
      function of $R$ and $\mu$. On the left, $200$ iterations of this process
      were run and the proportion of realisations that survived the first
      $250$ generations is shown. Dark blue colour corresponds to no surviving realisations
      and yellow to only surviving realisations. The right image zooms in
      the region of small $\mu $'s. In both
      cases the red line is our theoretical bound for extinction from
      Theorem~\ref{thm:R_fixed_extinction}.}
    \label{fig:phase_diagram}
\end{figure}

\begin{itemize}
  \item Is there a sharp transition? That is, for given $R$, is the survival
  region a (possibly empty) interval of values of $\mu$? See
  also Figure~\ref{fig:phase_diagram}.
  \item Is there always extinction for small values of $R$?
  Simulations suggest that in $d=1$ for $R \leq 2$ the process dies out
  for all values of $\mu$, see Figure~\ref{fig:phase_diagram} again.
   \item Can one give results for ``soft'' annihilation, allowing multiple occupancy of the sites?
  Of course, instead of the strong competition we consider, one could
  look at truncation, keeping for instance at most $N$ particles per site
  at the same time and removing the others. Theorem~1.1 in
  \cite{Mueller2015} implies for this truncation in our model that there
  is, for each $\mu > 1$ and all $R$, a critical value
  $N_c \in \{2,3,4, \ldots \}$ such that the survival probability is $0$
  for $N\leq N_c$ and strictly positive for $N> N_c$.
  \item What is the speed for the stochastic ``travelling waves'' in our model?
  Is there a shape theorem?  
  \item Are there ``discrete travelling waves'' for the coupled map lattice defined in~\eqref{eqn:cmlattice}? If so, how is their deterministic speed related to the speed of the stochastic ``travelling waves'' for the BARW? 
  \item
  The representation \eqref{eqn:next_gen_barw} suggests an interesting
  connection to spread-out oriented site percolation: let each site be
  open with probability $p$ and closed with probability $1-p$, where
  $p = \min \{ \varphi_{\mu}  ((2R+1)^{-d} ), \varphi_{\mu}(1) \}$.
  Connect the open sites at time $n+1$ to their ``parent'' (with distance
    $\le R$) at time $n$, provided it is open. Then the ``wet'' sites at
  time $n$ are a lower bound for $\eta_n$. \\ Let $p_c(d, R)$ be the
  percolation threshold for the event that there is an infinite connected
  cluster. How does the percolation threshold in directed space-time
  percolation behave for $R \to \infty$?\\ We have the following
  conjecture, based on the analogy with ``spread-out oriented bond
  percolation'', see \cite{vanderHofstadSakai2005}:
  \begin{equation*}
    \lim_{R\to\infty}(2R+1)^d p_c (d,R)=1
    \quad \text{for every } d >4.
  \end{equation*}
  It is plausible since the lattice should be more and more tree-like in
  high dimensions but we could not find a proof in the literature. Since
  $\varphi_\mu'(0)>1$, this conjecture would lead to an alternative proof
  of survival for large $R$ in $d >4$.
\end{itemize}

\begin{appendix}\label{appendix}
\section*{}

For completeness and ease of reference, we state the following
concentration estimate for sums of independent Bernoulli random
variables, which is a straightforward consequence of Bernstein's inequality.

\begin{lemma}
  \label{lem:Bernstein}
  Let $(X_i)_{i=1,\dots,n}$ be independent Bernoulli random variables
  with $p_i = \mathbb P(X_i =1)$, and let $S_n:=X_1+\cdots+X_n$. Then,
  setting $\mu_n := \E [ S_n ] = \sum_{i=1}^n p_i$,
  $ \sigma^2_n := \Var S_n =\sum_{i=1}^n p_i(1-p_i)$, and
  $m_n := \max_{1 \le i \le n} \max\{p_i,1-p_i\} = \max_{1 \le i \le n} \mathrm{ess\, sup} \abs{X_i - \E[X_i]}$
  ($\le 1$), we have
  \begin{align}
    \label{eqn:Bernsteinineq}
    \P( S_n - \mu_n \ge w )
    \le \exp\Big( - \frac{w^2}{2\sigma_n^2 + (2/3) m_n w} \Big),
    \quad w \ge 0,
  \end{align}
  and the same bound applies to $\P( S_n - \mu_n \le w )$ for $w \leq 0$.
\end{lemma}

\begin{proof}
  By Bernstein's inequality (see e.g.~\cite[Ineq.~(8)]{Bennett1962}), for
  every $t\ge 0$,
  \begin{align*}
    \P( S_n \ge \mu_n + t \sigma_n )
    \le \exp\Big( - \frac{t^2}{2+ 2 m_nt/(3\sigma_n)} \Big)
    = \exp\Big( - \frac{(\sigma_n t)^2}{2\sigma_n^2 + (2/3) m_nt \sigma_n}
      \Big).
  \end{align*}
  Reparametrising $t \sigma_n = w$ (and implicitly assuming $\sigma_n>0$,
    otherwise the problem becomes trivial) we can rewrite this as
  \eqref{eqn:Bernsteinineq}.

  Applying the argument to the $1-X_i$'s gives the same bound for
  $\P( S_n - \mu_n \le w )$, for $w \leq 0$.
\end{proof}

\end{appendix}

\begin{funding}
  The work was supported by Deutsche Forschungsgemeinschaft through DFG
  project no.\ 443869423 and by Schweizerischer Nationalfonds through SNF
  project no.\ 200021E\_193063 in the context of a joint DFG-SNF
  project within in the framework of DFG priority programme SPP 2265
  Random Geometric Systems.
\end{funding}

\smallskip
The authors would like to thank two anonymous referees and the associate editor who gave us constructive feedback that helped us to improve the presentation and correct inaccuracies.

\bibliographystyle{imsart-number}
\bibliography{my_bibliography}

\end{document}